\documentclass[11pt,english,reqno]{amsart}

\usepackage{amscd,amsfonts,amssymb,amsmath,amsthm,bbm}
\usepackage{graphicx,psfrag}
\usepackage[english]{babel}

\usepackage{palatino}
\usepackage[hypertex, colorlinks=true,
 linkcolor=blue, citecolor=blue]{hyperref}
\usepackage[T1]{fontenc}
\usepackage{mathrsfs}
\usepackage{amsmath,amsthm,amssymb,psfrag}
\usepackage{latexsym}

\usepackage{enumerate}
\usepackage{mathrsfs}
\usepackage{amsopn}
\usepackage{amsmath}
\usepackage{amssymb}
\usepackage{amsfonts}
\usepackage{amsbsy}
\usepackage{amscd,indentfirst,epsfig}
\usepackage{hyperref}
\usepackage{amsfonts,amsmath,latexsym,amssymb,verbatim,amsbsy}
\usepackage{amsthm}
\theoremstyle{plain}
\newtheorem{theo}{Theorem}[section]
\newtheorem{lemme}[theo]{Lemma}

\newtheorem{propo}[theo]{Proposition}
\newtheorem{cor}[theo]{Corollary}

\newtheorem{nb}[theo]{Remark}

\numberwithin{equation}{section}

\def \leq {\leqslant}
\def \geq {\geqslant}

\def\ind#1{\lower5pt\hbox{$\scriptstyle #1$}}
\def \d {\,\mathrm{d} }
\def \L {\mathcal{L}}

\def \H {\mathcal{H}}

\def \ds {\displaystyle}

\def \D {\mathcal{D}}

\def \ds {\displaystyle}
\def\Q {\mathcal{Q}}

\def\R{\mathbb{R}}

\def \n {n}

\def \u {U}
\def \v {{v}}

\def \vb {\v_*}

\def \M {\mathcal{M}}

\def \d {\mathrm{d }}
\def \D {\mathscr{D}}

 \def \S {\mathbb{S}^{d-1}}

\def \com {$C_0$-semigroup }

\def \ut {(\u(t))_{t \geq 0}}

\def \vt {(V(t))_{t \geq 0}}

\def \It {\int_{\R^d \times \S}}

\def \IR {\int_{\R^d}}

\def \X {\mathbf{X}}
\def \Y {\mathbf{Y}}

\title[]
{Convergence to equilibrium for linear spatially homogeneous Boltzmann equation with hard and soft potentials: a semigroup approach in  $L^1$-spaces}

\author{B. Lods \& M. Mokhtar-Kharroubi}

\address{Bertrand Lods, Dipartment of Economics and Statistics ESOMAS  \& Collegio
  Carlo Alberto, Universit\`{a} degli Studi di Torino,  Corso Unione
  Sovietica, 218/bis, 10134 Torino, Italy}
\email{bertrand.lods@unito.it}

\address{Mustapha Mokhtar-Kharroubi, Universit\'e de Franche-Comt\'e, Equipe de Math\'ematiques, CNRS UMR 6623, 16, route de Gray, 25030 Besan\c con Cedex, France
}
\email{mustapha.mokhtar-kharroubi@univ-fcomte.fr}

\begin{document}

\maketitle

\begin{abstract} We investigate the large time behavior of solutions to the spatially homogeneous linear Boltzmann equation from a semigroup viewpoint. Our analysis is performed in some (weighted) $L^{1}$-spaces. We deal with both the cases of hard and soft potentials (with angular cut-off). For hard potentials, we provide a new proof of the fact that, in weighted $L^{1}$-spaces with exponential or algebraic weights, the solutions converge exponentially fast towards equilibrium. Our approach uses weak-compactness arguments combined with  recent results of the second author on positive semigroups in $L^{1}$-spaces. For soft potentials, in $L^{1}$-spaces, we exploits the  convergence to ergodic projection for perturbed substochastic semigroup \cite{mmk1} to show that, for very general initial datum, solutions to the linear Boltzmann equation converges to equilibrium in large time. Moreover, for a large class of initial data, we also prove that the convergence rate is at least algebraic. Notice that, for soft potentials, no exponential rate of convergence is expected because of the absence of spectral gap.\\

\end{abstract}

\tableofcontents

\section{Introduction}

We investigate in the present work the large time behavior of solutions to the linear Boltzmann equation for both hard and soft potentials, under some cut-off assumption. We consider solutions in some weighted $L^{1}$-spaces and our approach is based on functional analytic results regarding positive semigroups in such spaces. To keep the presentation and results simple, we investigate only here the spatially homogeneous Boltzmann equation which exhibits already a quite rich behavior. Extension to spatially inhomogeneous problems is planned for future investigations. 

\subsection{The kinetic model}

Before entering the details of our results, let us formulate the problem we aim to address here: we shall consider the spatially homogeneous linear Boltzmann equation (BE in the sequel)
\begin{equation}
\label{eq:BE}
{\partial_t} f(t,v)=\L f(t,v), \qquad f(0,v)=f_0(v) \geq 0\end{equation}
in which  $\L$ is the linear Boltzmann operator given
by
\begin{equation}\label{linearoperator}
\L f=\Q(f,\M)\end{equation}
where $\Q(f,g)$ denotes the bilinear Boltzmann operator
\begin{equation}\label{bolt}
\Q(f,g)=
\It B(v-\vb, \sigma) \left(
f(\v')g(\vb')
-f(\v)g(\vb)\right)\d\vb \d\sigma\end{equation} where  $\v'$ and $\vb'$ are the pre-collisional
 velocities   which result, respectively, in $ \v $ and $\vb$ after elastic
 collision
  \begin{equation}
\v'=\dfrac{\v +\vb}{2} + \dfrac{|v -\vb|}{2}\sigma, \qquad \vb'=\dfrac{\v +\vb}{2} - \dfrac{|v -\vb|}{2}\sigma.
\end{equation}
Here $f$ and $g$ are nonnegative functions of the velocity variable $\v \in\R^d$ and $B(q,\n)$ is a nonnegative function .
We will
assume throughout this paper that the distribution function
$\mathcal{M}$ appearing in \eqref{linearoperator}  is a given   Maxwellian
function:
\begin{equation}\label{maxwe1}
\M(\v)=\dfrac{\varrho}{\big({2\pi \Theta}\big)^{d/2}}\exp
\left\{-\dfrac{|\v- \mathbf{u}|^2}{2\Theta}\right\}, \qquad \qquad \v
\in \R^d,
\end{equation}
where $\mathbf{u}  \in \R^d$ is the given bulk velocity  and $\Theta
>0$ is the given effective temperature of the host fluid. Notice that, by Galilean invariance, there is no loss of generality in assuming
$$\varrho=\Theta=1, \qquad \mathbf{u}=0.$$
We shall investigate in this paper several collision operators $\L=\L_B$
corresponding to various interactions collision kernels $B=B(v-\vb,\sigma)$. Typically, we shall consider the case
$$B(v-\vb,\sigma)=\Phi(|v-\vb|)\,b(\cos \theta), \qquad \cos \theta=\left\langle \dfrac{v-\vb}{|v-\vb|}, \sigma \right\rangle$$
where  $b\::\:[-1,1] \to \R^+$ and $\Phi\::\:\R^+ \to \R^+$ are measurable. We shall consider models with \emph{angular cut-off}, i.e. we assume
\begin{equation}\label{cutoff}
\ell_b:=\int_{\S} b(\cos \theta) \d\sigma< \infty \qquad \forall v,\vb  \in \R^d;\end{equation}
which, in turns, reads
$$\ell_{b}=|\mathbb{S}^{d-2}|\int_{-1}^1 \left(1-s^2\right)^{\frac{d-3}{2}} b(s) \d s < \infty.$$
Notice that, without loss of generality, we can assume that $\ell_b=1.$ Typically, we shall consider the case of power-like potentials $\Phi(\cdot)$:
$$\Phi(r)=r^\gamma \qquad \forall r > 0$$
where we distinguish between two cases:
\begin{enumerate}
\item \textit{\textbf{Hard potentials}} for which $\gamma \geq 0$;
\item \textit{\textbf{Soft potentials}} for which $-d < \gamma < 0.$
\end{enumerate}
A very peculiar model, particularly relevant for physical applications, is the one of \emph{hard-spheres} in dimension $d=3$   for which
$$B(v-\vb,\sigma)=B_{\mathrm{hs}}(\v-\vb,\sigma)= b_0 |v-\vb|$$
i.e. $\Phi(r)=r$ and $b(s)=b_0$ is constant.

Notice that, \emph{independently of the interaction kernel} $B$, it is very easy to check that $\L(\M)=0$ and that, in any reasonable spaces, the kernel of $\L$ is spanned by $\M$, i.e.
$\M$ is the only equilibrium with unit density.

Notice that, thanks to the cut-off assumption \eqref{cutoff}, the Boltzmann operator $\L$ can be split into a gain part and a loss part:
$$\L f= K f - T\,f$$
where the gain part
\begin{equation}\label{gain}
K f(v)=\L^+ f(v):=\It B(v-\vb, \sigma)
f(\v')\M (\vb')\d \vb \d \sigma \qquad v \in \R^d\end{equation}
while the loss part $T=\L^-$ is a multiplication operator
$$T f =\It  B(v-\vb, \sigma)
f(\v )\M (\vb)\d \vb \d \sigma=\Sigma(v)\,f(v)$$
where the collision frequency
$$\Sigma(v)=\It  B(v-\vb, \sigma)\M(\vb)\d\vb \d\sigma=\ell_b\, \left(\Phi \ast \M\right)(v)$$
where $\ast$ denotes the convolution in $\R^d$. For power-like potentials, i.e. whenever $\Phi(r)=r^\gamma$,  there exist two positive constants $\sigma_1,\sigma_2 > 0$ such that
\begin{equation}\label{sigma}
\sigma_1 \left(1+|v| \right)^\gamma \leq \Sigma(v) \leq \sigma_2 \left(1+|v|\right)^\gamma \qquad \forall v \in \R^d.
\end{equation}

The non-local part $K$ admits an integral form, based upon the Carleman representation of $\Q^+$. Namely, there exists some measurable kernel $k=k_B(v,w)$ such that
\begin{equation}\label{eq:Kkerne}
Kf=\L^+_B f(v)=\int_{\R^d} k_B(v,w)f(w)\d w \qquad \forall f \in L^1(\Sigma(v),\d v)=L^1\left(\left(1+|v|\right)^\gamma,\d v\right).\end{equation}
The explicit form of the kernel $k_{B}(v,w)$ can be derived, following the lines of \cite{carleman}. Details of such computations are given in Appendix \ref{sec:appA}.

\subsection{Presentation of the results and strategy}

We provide here an \emph{unified} approach to the problem of convergence towards equilibrium for solutions to \eqref{eq:BE}. As already said, we will deal with solutions to \eqref{eq:BE} in suitable weighted $L^{1}$-spaces. Namely, we shall consider here functional space of the type
$$\X=L^1(\R^{d}\,,\,m^{-1}(v)\d v)$$
endowed with its natural norm where  the weight function $m=m(v)$ positive and such that $m^{-1}(v) \geq 1$ for any $v \in \R^{d}$. We shall consider here two different types of weights:
\begin{align}
&\textbf{(Exponential weights)} \qquad m(v)=\exp(-a|v|^{s}) \qquad a \geq  0\,,\qquad s \in (0,1],\label{expwe}\\
\noindent\text{or} \nonumber \\ 
&\textbf{(Algebraic weights)} \qquad m(v)=\left(1+|v|^{\beta}\right)^{-1}, \qquad \beta \geq 0.\label{algwe}
\end{align}
 
More precisely, for hard-potentials, our approach will require the weights $m$ to be non trivial, i.e $a >0$ for exponential weights or $\beta >0$ for algebraic weight where as, for soft potentials, we will deal with ``unweighted'' spaces, i.e. $m(v)=1$ for any $v \in \R^{d}$ which corresponds to $a=0$ or $\beta=0.$
 
The approach we present here is, by some aspects, much simpler than the afore-mentioned ones and relies only on very general results about positive semigroups in $L^{1}$-spaces. 

We wish in particular to emphasize that one of the main interesting features of our approach is that provides directly both the well-posedness (in the semigroup sense) of \eqref{eq:BE} and the asymptotic behavior of the solutions. This is a major contrast with respect to the recent quantitative approach \cite{gualdani} based upon space enlargement and factorisation for which the well-posedness of the associated problem is always given as a first \emph{initial assumption}. 
%For this reason, the approach is purely qualitative and no indication of the speed of convergence can be deduced from our results. 
The strategy we adopt is very natural from the viewpoint of classical perturbation theory of semigroups. It consists, roughly speaking, in considering the non-local operator $K$ as a perturbation of the multiplication operator $T$ associated to $\Sigma$ and showing that  $K$ is \emph{weakly compact} with respect to $T$ in $\X$. Notice that, for  hard-potentials,  $K$ is an unbounded operator and one has to invoke perturbation theorem for semigroup which allows for unbounded perturbation. Such generation result is recalled in the next section as well as the basic results concerning positive semigroup in $L^{1}$-spaces.

For hard potential, our main result is the following:
\begin{theo}\label{theo:introHard}
If the weight function is given by \eqref{expwe} or \eqref{algwe}, then, for any $\gamma \in [0,d-2]$ and any $b \in L^{1}(\S)$, the linear Boltzmann operator $(\L,\D(\L))$ with 
$$\D(\L)=\Y=:L^{1}(\R^{d},\left(1+|v|\right)^{\gamma}\,m^{-1}(v)\d v) \subset \X$$
 is the generator of a positive \com $(V(t))_{t \geq 0}$ on $\X$. Moreover,  there exist $C > 0$ and $\lambda_\star >0$ such that the \com $(V(t))_{t \geq 0}$ generated by $(\L,\D(\L))$ in $\X$ satisfies
$$\|V(t)f_{0}-\varrho_0 \M \|_\X \leq C\exp(-\lambda_\star t)\|f_{0}\|_\X \qquad \text{ for all nonnegative }  f_{0}\in \X, \quad \forall t \geq 0$$
where $\varrho_0=\ds \int_{\R^d} f_{0}(v)\d v$ for any $f_{0} \in \X$.
\end{theo}

The above result asserts that \eqref{eq:BE} is well-posed in the semigroup sense and that the associated solutions converge exponentially fast towards the unique equilibrium with same mass as $f_{0}$. Notice that the convergence of $V(t)$ as $t \to \infty$ holds in \emph{operator norm}. The proof of the Theorem is perturbative in nature and consists in looking $\L$ and the associated semigroup $(V(t))_{t\geq 0}$ as perturbations of the multiplication operator $T$ and the associated semigroup $(U(t))_{t\geq0}$. Notice that the spectrum of $T$ is given by
$$\mathfrak{S}(T)=\overline{\mathrm{Range}(-\Sigma)}=(-\infty,-\eta]$$
where $\eta=\inf_{v \in \R^{d}}\Sigma(v)$ and $\overline{\mathrm{Range}(-\Sigma)}$ denotes the closure (in $\mathbb{C}$) of the range of $-\Sigma(\cdot).$ The main steps to prove the above Theorem are the following:

\begin{enumerate}[i)]
\item The main tool in the proof is the fact that the collision operator $K$ is positive and weakly compact as an operator from $\Y \to \X$. Since $\Y$ is also the domain of the multiplication operator $T$, it shows that $K$ is $T$-relatively weakly compact. In particular, this implies that the spectrum of $\L=K+T$ consists in some essential spectrum which is included in the spectrum of $T$ and eigenvalues of finite algebraic multiplicities which can accumulate only 
in $-\eta$.  
\end{enumerate}

In particular, this shows the existence of some spectral positive spectral gap $\lambda \in (0,\eta)$ for $\L$. Remember that, in general and because of the absence of some \emph{spectral mapping theorem} for semigroups, this is not enough in general to obtain the decay prescribed by Theorem \ref{theo:introHard}. However, since we are dealing with \emph{positive semigroups in} $L^{1}$-spaces, the decay can be obtained by the following arguments:

\begin{enumerate}[ii)]
\item From the above point and since we are dealing here with positive operators in $L^{1}$-spaces, Desch's theorem directly ensures that $\L=T+K$ is the generator of a positive semigroup $(V(t))_{t \geq 0}$ on $\X$ which is a perturbation of the \com $(U(t))_{t\geq 0}$ generated by $T$. Moreover, these two semigroups share the essential spectrum. 
\item Finally, since $0$ is an algebraically simple eigenvalue of $\L$  the decay estimate follows from general results concerning positive semigroups in $L^{1}$-spaces.
\end{enumerate}

In the above steps, the most technical part of the proof consists in showing the weak-compactness property of $K\::\:\Y \to \X$. To do so, we exploit the integral nature of $K$ given by \eqref{eq:Kkerne} together with suitable comparison techniques that allow us to restrict ourselves to the properties of $K$ whenever $\gamma=d-2$ (corresponding to hard-spheres interactions). Notice in particular that, for such a technical step, the introduction of the weight $m$ is definitely necessary since it can be shown that, without weight, the operator $K$ is not weakly-compact (see \cite{m2As} and Remark \ref{noweight}).\medskip

For soft potentials, our main contribution are the following two results. First, a non quantitative convergence result is establish 
\begin{theo}\label{theo:introSoft} 
Assume now that $\X=L^{1}(\R^{d},\d v)$. If  $\gamma \in (-d,0)$ and $b(\cdot) \in L^1(\S)$ then 
$$\X=\mathrm{Ker}(\L) \oplus \overline{\mathrm{Im}(\L)}$$
and the \com $(V(t))_{t \geq 0}$ generated by $\L$ in $\X$ satisfies
\begin{equation}\label{eq:vtf}
\lim_{t \to \infty}\|V(t)f-\varrho_f \M \|_\X =0\end{equation}
where $\varrho_f=\ds \int_{\R^d} f(v)\d v$ for any nonnegative $f \in \X.$
\end{theo}
Second, whenever $f \in \mathrm{Ker}(\L) \oplus \mathrm{Im}(\L)$, we can explicit the rate of convergence:
\begin{theo}\label{theo:introSoft1}
Under the assumptions of Theorem \ref{theo:introSoft}, given some nonnegative 
$$f \in \mathrm{Ker}(\L) \oplus \mathrm{Im}(\L),$$
then
 for any $c \in (0,1)$,  
$$\|V(t)f-\varrho_f\,\M\|_{\X}=\mathrm{O}\left(\vartheta_{\log}^{-1}(ct)\right) \qquad \text{ as } \quad t \to \infty$$
where $\vartheta_{\log}^{-1}$ is the inverse  of the increasing mapping  $\vartheta_{\log}\::\:r > 0 \mapsto \vartheta(r)\log \left( 1+\frac{\vartheta(r)}{r}\right)$  and
$$\vartheta(r)=\frac{1}{r}\dfrac{1}{1-\frac{\Sigma _{\max }}{\sqrt{r^{2}+\Sigma _{\max
}^{2}}}}, \qquad \Sigma_{\max}=\sup_{v\in \R^{d}}\Sigma(v).$$
In particular, for any $\varepsilon > 0$, there exists $C=C(f,\varepsilon) > 0$ such that
$$\|V(t)f-\varrho_f \M\|_{\X} \leq C\left(1+t\right)^{-\frac{1}{3+\varepsilon}} \qquad \forall t \geq 0.$$
\end{theo}

Again, the above result asserts that \eqref{eq:BE} is well-posed in the semigroup sense and that the associated solutions converge  towards the unique equilibrium with same mass as the initial datum. The proof of the above result is a  consequence of a general result  concerning the strong convergence to  ergodic projection for perturbed substochastic semigroups \cite{mmk1}. Notice that such a result relies on the so-called ``0-2'' law for $C_{0}$-semigroups by G. Greiner \cite[p. 346]{arendt} for the non-quantitative version \eqref{eq:vtf}. To get a quantitative convergence rate, we resort to recent results about non-exponential convergence rate for semigroups \cite{Chill} related to \emph{Ingham's theorem} (see also \cite{batty1,batty2}). Notice that, because of the absence of a spectral gap for $\L$, we do not expect any exponential relaxation to equilibrium.

\subsection{Related  literature}

As already mentioned, we investigate here both the cases of hard and soft potentials. As well-documented, at least in an Hilbert setting, these two cases exhibit very different behaviour and usually require different methods of investigation. Let us describe here the known results regarding the large time behavior of solutions to \eqref{eq:BE}. A first general comment is that, due to its paramount importance for the study of close-to-equilibrium solutions to the  nonlinear Boltzmann equation, the \emph{linearized} Boltzmann equation received much more attention than the \emph{linear} one. Let us recall here that the operator associated to the linearized Boltzmann equation is given by
$$\mathscr{L}f=\Q(f,\M)+\Q(\M,f)$$
whereas we recall that $\L f=\Q(f,\M).$ One sees in particular that both the operators are expected to share many spectral properties. However, two very important differences have to be emphasised: 
\begin{enumerate}
\item First, while the linear operator $\L$ admits a single equilibrium state with unit mass given by $\M$, the linearized operator $\mathscr{L}$ is such that
$$\mathscr{L}(\M)=\mathscr{L}(v_{i}\,\M)=\mathscr{L}(|v|^{2}\,\M)=0, \qquad \forall i=1,\ldots,d.$$
This means that, in any reasonable functional space, $0$ is a \emph{simple} eigenvalue of $\L$ with $\mathrm{Ker}(\L)=\mathrm{span}(\M)$ while the kernel of $\mathscr{L}$ is $d+2$-dimensional:
$$\mathrm{Ker}(\mathscr{L})=\mathrm{span}\{\M\,v_{i}\,\M\,\,|v|^{2}\,\M\,;\,i=1,\ldots,d\}.$$
\item Second and more important for the analysis we shall perform here, the  semigroup associated to the linear Boltzmann operator $\L$ is positive whereas the  one associated to $\mathscr{L}$ is not. This is very natural from the physical point of view since equation \eqref{eq:BE} is aimed to describe the evolution particles interacting with an host medium while the evolution equation associated to $\mathscr{L}$ aims to describe the \emph{fluctuation around the equilibrium} of the solution to the nonlinear Boltzmann equation. As already said,  we will take much benefit of the fact that the semigroup generated by $\L$ is positive.
\end{enumerate}

Let us now describe in more details the existing results devoted to the large time behaviour of solution to \eqref{eq:BE}. We shall distinguish here the two cases of hard and soft potentials.\medskip

$\bullet$ As far as hard potentials are concerned, it is by now well understood that, in the Hibert space 
\begin{equation}\label{eq:H}
\mathcal{H}=L^{2}(\R^{d},\M^{-1}(v)\d v)\end{equation}
the linear Boltzmann operator (with its natural domain) is negative, self-adjoint. This means that
$$\mathscr{E}(f,f):=-\int_{\R^{d}}\,f\,\L f\, \M^{-1}\d v \geq 0 \qquad \forall f \in \D(\L) \subset \mathcal{H}.$$
Moreover, $0$ is a simple eigenvalue associated to the eigenfunction $\M$ and there exists $\lambda_{2}>0$ such that
$$\mathscr{E}(f,f) \geq \lambda_{2} \|f\|_{\mathcal{H}} \qquad \forall f \in \D(\L), \qquad f \bot \M$$
where of course the orthogonality is meant with respect to the natural inner product of $\mathcal{H}.$ 

Clearly, this proves that $\L$ admits a spectral gap in $\mathcal{H}$ of size at least $\lambda_{2}.$ Since $\L$ is self-adjoint in $\mathcal{H}$, this spectral result for the generator $\L$ translates ``for free'' to the associated $C_{0}$-semigroup $\left(\exp(t\L)\right)_{t \geq 0}$ yielding an exponential trend to equilibrium with rate $\exp(-\lambda_{2}t).$ 
The first proof of such a result can be traced back to D. Hilbert himself \cite{hilbert}. General proof of such an exponential trend to equilibrium are available in the literature. Historically, the first proof is based upon Weyl's Theorem and consists in showing that the non-local operator $K$ is a compact perturbation of the multiplication operator $T$, ensuring that the essential spectrum of $\L$ coincide with that of $T$ (see \cite{grad, hilbert}). Only recently, such an approach has been made quantitative and explicit estimate of the spectral gap $\lambda_{2}$ have been provided \cite{LMT} for hard-spheres interactions. 

The above approach seems at first sight purely hilbertian and not well-adapted to the study of \eqref{eq:BE} in  spaces of physical interest like $L^{1}$-spaces. In such spaces indeed, the problem has received much less attention. We mention here the seminal work of Suhadolc \cite{suha} who introduced the spectral study of $\L$ for hard-spheres interactions in weighted spaces like
$$L^{1}(\R^{d},\exp(\alpha |v|^{2})\d v).$$
In particular, it can be proven then that, in such a space, $\L$ is the generator of a positive semigroup. More recently, a careful study of the \emph{linearized Boltzmann operator} has been performed in weighted $L^{1}$-spaces with stretched exponential weights \cite{Mo}. Such an approach is based upon the knowledge of the spectral properties of the linearized operator in the Hilbert space $\mathcal{H}$ combined with recent results concerning the stability of the spectrum in enlarged functional spaces \cite{gualdani}. This approach has been applied recently for the linear Boltzmann operator with hard-spheres interactions in \cite{BCL} and provides quantitative estimates of the spectral gap in weighted $L^{1}$-spaces. Such results are presented in the Appendix \ref{sec:appB} of the paper for the sake of completeness. We also mention  the recent result \cite{BCL2} in which the exponential convergence of the solutions to \eqref{eq:BE} is measured in \emph{relative entropy} and an explicit convergence rate is provided by means of entropy/entropy production estimate.\medskip

$\bullet$ The  linear Boltzmann equation associated to cut-off soft potentials received much less attention. In the seminal paper \cite{caflisch}, the \emph{linearized} operator has been studied in the space $\mathcal{H}$ defined by \eqref{eq:H}. Adapting the argument of \cite{caflisch} to the linear operator $\L$, one checks easily that, in this case, the linear operator $\L$ is still self-adjoint and negative in $\H$ but then its continuous spectrum extends to the origin so that $\L$ has \emph{no spectral gap} in the space $\mathcal{H}.$ It is then clear that, at the semigroup level, no exponential convergence towards equilibrium can occur. We mention here that, in \cite{caflisch}, almost exponential convergence to equlibrium is obtained for the linearized Boltzmann for initial datum belonging to the subspace of $\H$ made of functions bounded above by some suitable Maxwellian (see Remark \ref{nb:caflisch} for more details). Let us mention here that, using argument similar to those of Theorem \ref{theo:introSoft1}, it is possible to obtain algebraic rate of convergence (like $t^{-1}$) for the linear Boltzmann operator in the space $\H$ (see Remark \ref{nb:caflisch}). To our knowledge, the only convergence results available for the linear Boltzmann equation are all based upon entropy inequalities \cite{pettersson} and consists in showing that, since the relative entropy (with respect to the equilibrium solution $\M$) is continuously decreasing along the solutions to \eqref{eq:BE} and since $\M$ is the only equilibrium solution (with unit mass), a variant of classical LaSalle invariance principle implies the convergence of $f(t,v)$ towards $\M$. Such an approach requires at least that the initial state $f_{0}$ is of finite relative entropy, i.e. 
\begin{equation}\label{entropy}
\int_{\R^{d}}f_{0}(v)\log\left(\frac{f_{0}(v)}{\M(v)}\right)\d v < \infty\end{equation}
which, roughly speaking, would correspond  to convergence of $f(t,v)$ in some suitable Orlicz-space $L\log L.$

\subsection{Organization of the paper} After this long introduction, the paper is organized as follows. Section \ref{sec:remind} contains all the material from semigroup theory and spectral analysis needed for the rest of the paper. The results presented there are not new but maybe some of them are not very well-known in the kinetic community. We emphasise in particular the crucial role played by positivity in our analysis. Section \ref{sec:hard} is devoted to the study of the linear BE in the case of hard potentials and is culminating with the proof of Theorem \ref{theo:introHard} following the above steps i)--iii) here above. Section \ref{sec:soft} deals with the convergence result for soft potential interactions and provides the proof of Theorem \ref{theo:introSoft}. In Appendix \ref{sec:appA}, we derive the expression of the kernel $k_{B}(v,w)$ for various type of interactions $B$ as well as several of their important properties. The material of this Appendix is the basic toolbox for the proof of the weak compactness of $K$ used in the proof of Theorem \ref{theo:introHard}. Finally, in Appendix \ref{sec:appB}, we provide some quantitative estimates of the spectral gap $\lambda_{\ast}$ appearing in Theorem \ref{theo:introHard} using some of the results established recently in \cite{BCL} which use the enlargement and factorisation techniques of \cite{gualdani}.

\section{Reminder about positive semigroups in $L^1$-spaces}\label{sec:remind}

We review here some known result about the asymptotic properties of stochastic semigroups as $t \to \infty$ in $L^{1}$-spaces. We will focus here only on the  asymptotic stability of positive semigroups, i.e. the convergence towards equilibrium density. For such questions, one should distinguish between two important cases: a first one for which the existence of spectral gap implies exponential convergence towards equilibrium; a second one in which the generator does not exhibit a spectral which is related to ergodic theorem. We begin with general properties of positive semigroups in $L^{1}$-spaces.

\subsection{Spectral
properties of semigroups in $L^1$} Let us consider a Borel measure $\mu$ over some given space $\Omega$ and let $X=L^1(\Omega,\d\mu).$ Given a nonnegative mapping
$$M\::\:\Omega \to \mathbb{R}$$
the multiplication operator
$$A\::\:\D(A) \subset X \to X$$
defined by
$$Af(x)=-M(x)f(x) \quad \text{ for any } \quad f \in \D(A)=\left\{f \in X\,;\,M f \in X\right\}$$ is the generator of a \emph{positive} $C_0$-semigroup $\ut$ in $X$ given by
$$[U(t)f](x)=\exp(-M(x)t)f(x) \qquad \forall t \geq 0.$$
\begin{nb}
The semigroup is positive in the sense that it preserves the positive cone of $X$, i.e. $f \in X$, $f \geq 0$ $\mu$-a.e. implies $U(t)f \geq 0$ $\mu$-a.e.
\end{nb}
\begin{nb}
Notice that, if equipped with the graph norm $\|f\|_A:=\|f\| + \|Af\|$, the space $\D(A)$ can be identified with $L^1(\Omega,(1+M(x))\d\mu(x)).$
\end{nb}
General result about semigroup theory ensures that, as a multiplication operator, the spectrum of $A$ is given by the so-called essential range of $-M$, i.e.
$$\mathfrak{S}(A)=\left\{\lambda\in \mathbb{C}\,;\,\mu\left(\left\{x \in \Omega\,,\,|\lambda+M(x)| < \varepsilon \right\}\right)\neq 0 \forall \varepsilon >0\right\}$$
and, if $M(\cdot)$ is continuous, one actually has
$$\mathfrak{S}(A)=\mathrm{Range}(-M)\subset \R_{-}.$$
Since, for any $t \geq 0,$ $U(t)$ is also a multiplication operator, its spectrum is given by the essential range of $\exp(-M(\cdot) t)$ and, in particular,
$$\mathfrak{S}(U(t))=\overline{\exp\left(\mathfrak{S}(A)t\right)} \qquad \forall t \geq 0$$
where the closure is meant of course in $\mathbb{C}$. In particular, if we define
$$M_{0}:=\mu-\text{ess-inf}_{x \in \Omega} M(x),$$
then, then one has
$$\mathfrak{S}(A) \subset (-\infty,-M_{0}]$$
and the type of the semigroup $\ut$ is less than $-M_{0}$. Recall the type of a semigroup $\ut$ is defined as
$$\omega_0(U)=\inf\{\omega \in \mathbb{R}\,,\,\exists C \geq 1\,\,\|U(t)\| \leq C \exp(\omega t) \qquad \forall t\geq 0\}.$$

A fundamental property of positive semigroups in $L^{1}$-spaces is the following (see \cite{engel})
\begin{propo}\label{propo:sLomega}
If $X=L^{1}(\Omega,\mu)$ and   $(S(t))_{t\geq 0}$ is a positive \com of $X$ then its type coincide with the spectral bound of its generator $G$, i.e. $$\omega_{0}(S)=s(G).$$\end{propo}
 
Let us consider now the influence of unbounded perturbations. Namely, consider now a positive (integral) operator $B \::\:\D(A) \subset X \to X$ and assume that, for any $\lambda > -M_{0}$,
$$BR(\lambda,A)\::\:X \to X \qquad \text{ is a weakly compact operator in } X$$
where  $R(\lambda,A)=(\lambda-A)^{-1}$ is the resolvent of $A$ which exists at least for $\lambda > -M_{0}$. In such a case, even if $B$ is not a bounded operator, one knows 'for free' that the sum $A+B$ is the generator of a positive \com in $X$. Precisely, one has:
\begin{theo}{\cite[Theorem 6]{MMK}}\label{theommk}  If $B$ is a positive operator such that $BR(\lambda,A)$ is a weakly compact operator in $X$ then, $T=A+B$ (with domain $\D(A)$) is the generator of a \com $\vt$ in $X$. Moreover, the essential spectrum of $\vt$ is such that
$$\mathfrak{S}_{\mathrm{ess}}(U(t))=\mathfrak{S}_{\mathrm{ess}}(V(t)) \qquad \forall t \geq 0$$
where we adopted here the notion of essential spectrum due to Schechter.
\end{theo}
\begin{nb} Since $X$ is an  $L^{1}$-spaces, for an (un)bounded operator $T\::\:\D(T) \subset X \to X$ the essential spectrum of $S$ is exactly the part of the spectrum left invariant by weakly-compact operators, i.e.
$$\mathfrak{S}_{\mathrm{ess}}(T)=\bigcap_{K \in \mathscr{W}(X)} \mathfrak{S}(T+K)$$ 
where $\mathscr{W}(X)$ is the ideal of $\mathscr{B}(X)$ made of the weakly compact operators in $X$. We refer to \cite{latrach} for more details on
that matter.\end{nb}
\begin{nb} The semigroup $\vt$ in the above theorem is given by a Dyson-Phillips series
$$V(t)=\sum_{j=0}^\infty U_j(t),\qquad t \geq 0$$
where $U_0(t)=U(t)$ is the unperturbed semigroup and, for any $j \geq 0$,
\begin{equation}\label{dyson}
U_{j+1}(t)=\int_0^t U_j(t-s)BU_0(s) \d s.\end{equation}
\end{nb}

\begin{nb}Since $B$ is a weakly-compact perturbation of $A$, one has
$\mathfrak{S}_\mathrm{ess}(B)=\mathfrak{S}_\mathrm{ess}(A) \subset(-\infty,-M_{0}]$
and
$$\mathfrak{S}(B)\subset(-\infty,-\eta] \cup \{\lambda_k\}_k$$
where $(\lambda_k)_k$ is a sequence of eigenvalues of $B$ (maybe finite or empty). In particular, $\exp(\lambda_k t)$ is an eigenvalue of $V(t)$ for any $t \geq 0$ and any $k$.
\end{nb}

\subsection{Exponential trend to equilibrium in presence of a spectral gap} We begin here with a general result which holds for any \com in some Banach space $X$.
\begin{theo}{\cite[Theorem 3.7, Chapter V]{engel}}\label{theo:engel}
Let $\vt$ be a \com with generator $A$ in some Banach space $X$ such that 
$$\omega_{\mathrm{ess}}(V) < 0 $$
where $\omega_{\mathrm{ess}}(V)$ denotes the essential type of $\vt$.
Then the  set $\{\lambda \in \mathfrak{S}(A)\,;\,\mathrm{Re}\lambda \geq 0\}$  is finite (or empty) and consists of isolated eigenvalues $\lambda_{1},\ldots,\lambda_{m}$  of $A$ with finite algebraic multiplicity. If $\mathbb{P}_{1},\ldots,\mathbb{P}_{m}$ denote the corresponding spectral projections and $k_{1},\ldots,k_{m}$ the corresponding orders of poles of $R(\cdot,T)$ then
\begin{equation}\label{eq:vtrm}
V(t)= V_{1}(t) + \ldots + V_{m}(t) + R_{m}(t), \qquad t \geq 0\end{equation}
where
$$V_{n}(t)=\exp(\lambda_{n}t)\sum_{j=0}^{k_{n}-1}\dfrac{t^{j}}{j!}\left(A-\lambda_{n}\right)^{j}\mathbb{P}_{n} \qquad n=1,\ldots,m$$
and there exists $\delta > 0$ and $C \geq 1$ such that
$$\|R_{m}(t)\|_{\mathscr{B}(X)} \leq C \exp(-\delta t) \qquad \forall t \geq 0.$$
\end{theo}

The above result is important in particular when the above set $\{\lambda \in \mathfrak{S}(A)\,;\,\mathrm{Re}\lambda \geq 0\}$ reduces to a singleton $\lambda_{1}$ which becomes then the dominant eigenvalue of $T$. In such a case, the conclusion reads
$$\|V(t)-V_{1}(t)\| \leq C\exp(-\delta t) \qquad \forall t \geq 0.$$
If moreover $\lambda_{1}$ is a simple eigenvalue, i.e. $k_{1}=1$, then the rescaled semigroup $\exp(-\lambda_{1}t)V(t)$ converges exponentially fast towards a rank-one projection. Sufficient conditions for this to holds are explicit whenever $X$ is a Banach lattice and when the \com $\vt$ is positive. For simplicity, we state here our result in the case in which $X$ is some $L^{1}$-space:

\begin{propo}\label{prop:positive} Assume $X=L^{1}(\Omega,\mu)$ where $(\Omega,\mathcal{F})$ is a measure space and $\mu$ a given measure over $(\Omega,\mathcal{F}).$ Let $\vt$ be a positive \com with generator $A$ in $X$ satisfying the assumption of Theorem \ref{theo:engel}. Then, there exists $\varepsilon > 0$, 
$$\mathfrak{S}(A) \cap \{z \in \mathbb{C}\,;\,\mathrm{Re}z \geq s(A)-\varepsilon\}=\{s(A)\}$$
where $s(A)$ is an isolated eigenvalue of $A$ with finite algebraic multiplicity $k_{1}$ and the splitting \eqref{eq:vtrm} holds for $m=1$ and any $s(A)-\varepsilon < \delta < s(A)$. 
\end{propo}

Again, the above can be made more precise if the semigroup is irreducible. We recall that, in $L^{1}(\Omega,\mu)$, a \com $\vt$ is said to be irreducible if, given $f \in X$, $\phi \in X^{\star}$, $f, \phi > 0$ $\mu$-a. e., then there exists $t_{0} \geq 0$ such that
$$\langle V(t_{0})f\,,\,\phi\rangle_{X,X^{\star}} > 0$$
where $\langle \cdot,\cdot\rangle_{X,X^{\star}}$ denotes the duality bracket between $X$ and its dual space $X^{\star}=L^{\infty}(\Omega,\mu)$. 

One has then the following:
\begin{propo}{\cite[C.III. Proposition 3.5, p. 310]{arendt}}\label{aren}
Assume $X=L^{1}(\Omega,\mu)$ where $(\Omega,\mathcal{F})$ is a measure space and $\mu$ a given measure over $(\Omega,\mathcal{F}).$ Let $\vt$ be a \emph{irreducible} \com with generator $A$ in $X$. Then the following holds
\begin{enumerate}
\item Every positive eigenvector $f$ of $A$ is a quasi-interior point, i.e. $f > 0$ $\mu$-a. e.
\item Every positive eigenvector of the adjoint $A^{\star}$ is strictly positive.
\item If $\mathrm{Ker}(s(A)-A^{\star})$ contains a positive eigenelement, then $\mathrm{dim}(\mathrm{Ker}(s(A)-A)) \leq 1.$
\item If $s(A)$ is a pole of the resolvent of $A$, then it has algebraic multiplicity one. The corresponding spectral projection is then of the form $\Pi=\phi \otimes f$ where $\phi \in X^{\star}$ is a positive eigenvector of $A^{\star}$, $f \in X$ is a positive eigenvector of $A$ and $\langle u,\phi \rangle_{X,X^{\star}}=1$.
\end{enumerate}\end{propo}
\begin{nb}\label{nb:cyclic}
Notice that, if $\vt$ is irreducible with generator $A$, then there exists $\kappa \geq 0$ such that
$$\left\{\lambda \in \mathfrak{S}(A)\,;\,\mathrm{Re}\lambda=s(A)\right\}=s(A)+\kappa\,i\,\mathbb{Z}.$$
We refer to \cite[Theorem 1.12, Chapter VI]{engel} for further details.
\end{nb}
%\textbf{Comments: Put here some results about the existence of spectral gap, convergence to equilibrium and so on, \cite{mmk2}...}

\subsection{Convergence to equilibrium without spectral gap}  
In the above section, we described the situation in which the generator of the semigroup admits a spectral gap and its consequences on the exponential convergence. In this section, we describe a completely different situation in which no spectral gap estimate is available. In such a case, we do not expect the semigroup to converge exponentially to some ``equilibrium''. We present here some recent result, due to the second author, which provides in $L^{1}$-spaces sufficient condition for the semigroup to converge as time goes to infinity in absence of spectral gap. Such result will be then applied to the BE for soft potential in the last part of the paper.

\begin{theo}\label{02law}
Let $(U(t))_{t \geq 0}$ be a substochastic \com  in $L^{1}(\Omega,\mu)$ with generator $A$ and let 
$$B\::\:\D(A) \to L^{1}(\Omega,\d\mu)$$
be positive and satisfy 
$$\lim_{\lambda \to +\infty}\,r_{\sigma}\left(B(\lambda-A)^{-1}\right) < 1 \qquad \text{ and } \quad \int_{\Omega} Af+Bf \d\mu \leq 0 \qquad \forall f \in \D(B)\,;\, f \geq 0$$
where $r_{\sigma}(\cdot)$ denotes the spectral radius. Let $(V(t))_{t \geq 0}$ be the substochastic \com generated by
$$T=A+B\::\:\D(T)=\D(A) \to L^{1}(\Omega,\d\mu).$$
We assume that $(V(t))_{t \geq 0}$ is irreducible and that $\mathrm{Ker}(T)\neq\{0\}$. If
$$\lim_{t \to 0}\|U(t)f\|_{L^{1}(\Omega,\mu)}=0 \qquad \forall f \in L^{1}(\Omega,\mu)$$
and if there exists $m\in \mathbb{N}$ such that the mapping
$$t \geq 0 \mapsto R_{m}(t)=\sum_{j=m}^{\infty}U_{j}(t) \in \mathscr{B}(L^{1}(\Omega,\d\mu))$$ is continuous in operator norm (where $(U_{j}(t))_{j}$ has been defined in \eqref{dyson}), then
$$\lim_{t \geq 0} V(t)f=\mathbb{P}f$$  
where $\mathbb{P}$ denotes the ergodic projection on $\mathrm{Ker}(A).$
\end{theo}
\begin{nb}
Recall that, under the above hypothesis, $\mathrm{Ker}(A)=\mathrm{span}(\varphi)$ is one dimensional with $\varphi > 0$, $\mathrm{Ker}(A^{*})=\mathrm{span}(\varphi^{*})$ is one dimensional where $A^{*}$ denotes the adjoint of $A$ and $\varphi^{*} > 0$. In such a case, the \com $(V(t))_{t \geq 0}$ is mean ergodic with ergodic projection $\mathbb{P}$ given by
$$\mathbb{P}f =\left(\int_{\Omega} \varphi^{*}\,f \d\mu\right)\,\varphi \qquad \forall f \in L^{1}(\Omega,\d\mu)$$
where $\varphi^{*}$ and $\varphi$ are chosen such that $\int_{\Omega}\varphi^{*}\,\varphi\d\mu=1$ (see \cite{mmk1} for details).
\end{nb}

\section{Exponential trend to equilibrium for hard potentials}\label{sec:hard}

We consider in this section the case of hard-potentials and aim to prove Theorem \ref{theo:introHard}. We recall that we shall consider here a positive weight function $m=m(v)$ such that $m^{-1}(v) \geq 1$ for any $v \in \R^{d}$ and will work in the space
$$\X=L^1(m^{-1}(v)\d v).$$
The typical weights we shall consider are  exponential weights given by \eqref{expwe} and algebraic weights given by \eqref{algwe}. 

Since we are interested here in the properties of the linear BE in the space $\X$ for hard potential interactions, we shall consider the collision interactions
$$B(v-\vb,\sigma)=|v-\vb|^\gamma\,b(\cos \theta), \:\:\:\cos \theta=\left\langle \dfrac{v-\vb}{|v-\vb|};\n\right\rangle$$
with $b$ nonnegative satisfying the cut-off assumption \eqref{cutoff} and moreover $\gamma \geq 0.$ Introduce then
$$\Y=\Y_\gamma=L^1\left(\R^d\,;\,\left(1+|v| \right)^\gamma m^{-1}(v)\d v\right) \subset \X$$
and  let
$$T f(v)=\Sigma(v) f(v), \qquad \D(T)=\{f \in \X\;;\Sigma f \in \X\}=\Y_\gamma.$$
Clearly, $T$ generates a \com $(U(t))_{t \geq 0}$ in $\X$ defined by
$$U(t)f(v)=\exp(-\Sigma(v) t)f(v)\qquad t\geq 0,\:\:f \in \X.$$
The linear Boltzmann operator is then defined as
$$\L f(v)=Kf(v)-\Sigma(v)f(v)=\left(K+T\right)f(v), \qquad \forall f \in \D(T)$$
where the gain part $K$ is given by \eqref{gain}. 

A first very important observation is the following:
\begin{lemme}\label{lemme:app}
The linear operator $K$ depends continuously on $b \in L^{1}(\S)$: more precisely, if $(b_{n})_{n} \subset L^{1}(\S)$ is a sequence of kernels such that
$$\lim_{n\to \infty}\|b_{n}-b\|_{L^{1}(\S)}=\lim_{n \to \infty}\int_{\S}\left|b_{n}(\cos \theta)-b(\theta)\right|\d\sigma=0$$
then
$$\lim_{n \to \infty}\|K_{n}-K\|_{\mathscr{B}(\Y,\X)}=0$$
where $(K_{n})_{n}$ is the sequence of collision operators associated to $(b_{n})_{n}$.
\end{lemme}
\begin{proof} The proof relies on standard computations for the Boltzmann operator (see for instance \cite[Proposition 2.1]{Mo} for similar considerations). Let $\gamma \geq  0$ be given and let $(b_{n})_{n}\subset L^{1}(\S)$ be given.  Using the explicit expression \eqref{gain}, for any $n \in \mathbb{N}$ and any $f \in \Y$ it holds
$$
\|K_{n}f-Kf\|_{\X} \leq \int_{\R^{2d}\times \S}\left|B(v-\vb,\sigma)-B_{n}(v-\vb,\sigma)\right|\,|f(v')|\,\M(\vb')m^{-1}(v)\d v\d\vb\d\sigma$$
where we set $B_{n}(v-\vb,\sigma)=|v-\vb|^{\gamma}b_{n}(\cos\theta),$ $B(v-\vb,\sigma)=|v-v|^{\gamma}b(\cos\theta).$ Using the pre-post collisional change of variables $(v',\vb',\sigma) \to (v,\vb,\sigma)$ (see \cite[Chapter 1, Section 4.5]{villani}), we obtain
\begin{multline*}
\|K_{n}f-Kf\|_{\X} \leq \int_{\R^{2d}\times \S}\left|B(v-\vb,\sigma)-B_{n}(v-\vb,\sigma)\right|\,|f(v)|\,\M(\vb)m^{-1}(v')\d v\d\vb\d\sigma\\
= \int_{\R^{2d}}|b_{n}(\cos\theta)-b(\cos\theta)|\,|v-\vb|^{\gamma}|f(v)|\,\M(\vb)m^{-1}(v')\d v\d\vb\d\sigma\\
\leq \int_{\R^{d}}(1 +|v|)^{\gamma}\,|f(v)|\d v \int_{\R^{d}}\M(\vb)(1+|\vb|)^{\gamma}\d\vb\\
\int_{\S}\left|b_{n}(\cos\theta)-b(\cos \theta)\right|\,m^{-1}(v')\d\sigma
\end{multline*}
where we used that $|v-\vb|^{\gamma} \leq (1+|v|)^{\gamma}(1+|\vb|)^{\gamma}$ for any $v,\vb \in \R^{2d}.$ \\

Now, if $m$ is an exponential weight of type \eqref{expwe}, since $|v'|^{s}=(|v'|^{2})^{s/2}\leq (|v'|^{2}+|\vb'|^{2})^{s/2})=(|v|^{2}+|\vb|^{2})^{s/2} \leq |v|^{s}+|\vb|^{s}$ for any $v,\vb,\sigma$, we get that
$$m^{-1}(v') \leq m^{-1}(v)m^{-1}(\vb) \qquad \forall v,\vb,\sigma \in \R^{2d}\times \S$$
to that the above inequality reads
$$\|K_{n}f-Kf\|_{\X} \leq \|b_{n}-b\|_{L^{1}(\S)}\int_{\R^{d}} (1+|v|)^{\gamma}\,|f(v)|m^{-1}(v)\,\d v \int_{\R^{d}}\M(\vb) (1+|\vb|)^{\gamma}m^{-1}(\vb)\,\d\vb$$
and, since $\ds \int_{\R^{d}}\M(\vb) (1+|\vb|)^{\gamma}m^{-1}(\vb)\,\d\vb=C(m)< \infty$ we get that
$$\|K_{n}f-Kf\|_{\X}\leq C(m) \|b_{n}-b\|_{L^{1}(\S)}\,\|f\|_{\Y} \qquad \forall f \in \Y$$
which proves the result. \\

If the weight $m$ is an algebraic weight of type \eqref{algwe}, as above we see that, for any $v,\vb,\sigma \in\R^{2d}\times \S$,
$$m^{-1}(v')=1+|v'|^{\beta} \leq 1+\left(|v|^{2}+|\vb|^{2}\right)^{\beta/2} \leq 1+2^{\beta/2-1}\left(|v|^{\beta}+|\vb|^{\beta}\right)$$
so that
$$m^{-1}(v') \leq C_{\beta}\,m^{-1}(v)\,m^{-1}(\vb)$$
with $C_{\beta}=\max(1,2^{\beta/2-1})$. One concludes as above that 
$$\|K_{n}f-Kf\|_{\X}\leq C(\beta) \|b_{n}-b\|_{L^{1}(\S)}\,\|f\|_{\Y} \qquad \forall f \in \Y$$
for some positive constant $C(\beta)$ depending only on $\beta.$
\end{proof}

\begin{nb} As we shall see later on, since the angular collision kernel maybe approximated by some sequence of \emph{bounded} kernels over $\S$, to investigate the weak-compactness properties of $K$, it will be enough to assume $b$ to be bounded. \end{nb} 

 For simplicity, we denote $\mathbf{K}_\gamma$ the gain part associated to the \emph{variable hard-spheres} collision kernel $B_\gamma(v-\vb,\sigma)=|v-\vb|^\gamma,$ i.e.
$$\mathbf{K}_\gamma f(v)=\int_{\R^d \times \S} |\v-\vb|^\gamma f(\v')\M(\vb')\d\vb\d\sigma \qquad \forall f \in \Y$$
and, clearly, for any nonnegative $f$, one sees that
$$K f(v) \leq \|b\|_{L^\infty} \mathbf{K}_\gamma f(v) \qquad \forall v \in \R^d.$$
One has the following fundamental result
\begin{propo}\label{prop:comp} If the weight function $m$ are given by \eqref{expwe} or \eqref{algwe}
then, for any $\gamma \in [0,d-2]$,
$$\mathbf{K}_\gamma \::\:\Y \to \X$$
is a  positive weakly-compact operator.  \end{propo}
\begin{proof} Recall that (see Appendix A)
$$\mathbf{K}_{\gamma} f (v)=\int_{\R^d} k_{\gamma}(v,w)f(w)\d w$$
where $k_\gamma(v,w) \leq |v-w|^{\gamma-(d-2)}k_{d-2}(v,w)$ and
$$k_{d-2}(v,w)=C|v-w|^{-1}\exp\left(-\frac{1}{8}\left[|v-w| + \frac{|v|-|w|}{|v-w|}\right]^2\right) \qquad \forall v,w \in \R^d \times \R^d.$$
 Let us prove the weak compactness property. Let $\d\nu(v)=m^{-1}(v)\d v$ and let $\mathcal{B}$ be the unit ball of $\Y$. Since $\X=L^1(\d\nu)$, according to Dunford-Pettis Theorem, the weak compactness of $\mathbf{K}_{d-2}$  amounts to prove that
\begin{equation}\label{DP1}
\sup_{f \in \mathcal{B}}\int_{A} \left|\mathbf{K}_{\gamma} f(v)\right|\d\nu(v) \longrightarrow 0 \quad \text{ as } \quad  \nu(A) \to 0
\end{equation}
and
\begin{equation}\label{DP2}
\sup_{f \in \mathcal{B}}\int_{|v| > r} \left|\mathbf{K}_{\gamma} f(v)\right|\d\nu(v) \longrightarrow 0 \quad \text{ as } \quad r \to \infty.
\end{equation}
Using the representation of $\mathbf{K}_{\gamma}$ as an integral operator, it is easy to check that \eqref{DP1} and \eqref{DP2} will follow if one is able prove that
\begin{equation}\label{DP11}
\sup_{w \in \mathbb{R}^d}\dfrac{m(w)}{(1+|w|)^{\gamma}}\int_{A} k_{\gamma}(v,w) m^{-1}(v)\d v \longrightarrow 0 \quad \text{ as } \quad \nu(A) \to 0
\end{equation}
and
\begin{equation}\label{DP22}
\sup_{w \in \mathbb{R}^d}\dfrac{m(w)}{(1+ |w|)^{\gamma}}\int_{|v| > r} k_{\gamma}(v,w)  m^{-1}(v)\d v  \longrightarrow 0 \quad \text{ as } \quad r \to \infty.
\end{equation}
Let us prove \eqref{DP11}. Let $A \subset \mathbb{R}^d$ be a given Borel subset and let $w \in \mathbb{R}^d$ be fixed. Let $B_w=\{v \in \mathbb{R}^d\;,\;|v-w| < 1\}.$ Since $k_{\gamma}(v,w)  \leq C |v-w|^{\gamma+1-d}$ one has
\begin{multline*}
\int_{A} k_{\gamma}(v,w)  m^{-1}(v)\d v \leq C \int_A |v-w|^{\gamma+1-d}\d\nu(v)\\
=C \left(\int_{A \cap B_w} |v-w|^{\gamma+1-d}\d\nu(v)+\int_{A \cap B_w^c}|v-w|^{\gamma+1-d}\d\nu(v)\right).
\end{multline*}
Clearly, since $\gamma+1-d \leq -1$
$$\int_{A \cap B_w^c}|v-w|^{\gamma+1-d}\d\nu(v) \leq \nu(A)$$
while, for any $p >1$, $1/q+1/p=1$, one has
$$
\int_{A \cap B_w} |v-w|^{\gamma+1-d}\d\nu(v) \leq \left(\int_{A \cap B_w} \d\nu(v)\right)^{1/q}\,\left(\int_{A \cap B_w} |v-w|^{p(\gamma+1-d)}m^{-1}(v)\d v\right)^{1/p}.$$
Now, if $m(v)=\exp(-a|v|^{s})$ for $a > 0$ and $s \in (0,1)$, one gets that $$m^{-1}(v) \leq m^{-1}(w) \exp(a|v-w|^s) \leq \exp(a) m^{-1}(w) \qquad \forall w \in \mathbb{R}^{d}, \qquad \forall v \in B_{w}$$
while, for algebraic weights, i.e. if $m(v)=(1+|v|)^{-\beta}$, $\beta > 0$ then 
\begin{multline*}
m^{-1}(v)=(1+|v-w+w|^{\beta}) \leq 1+2^{\beta-1}|v-w|^{\beta}+2^{\beta-1}|w| \\
\leq (1+2^{\beta-1}) (1+|w|^{\beta})
 \qquad \forall w \in \mathbb{R}^{d}, \quad v \in B_{w}.\end{multline*}
One sees from this that, in both cases, there exists some positive constant $C > 0$ such that
$$\int_{A \cap B_w} |v-w|^{p(\gamma+1-d)}m^{-1}(v)\d v
\leq \,C\,m^{-1}(w)\int_{B_w} |v-w|^{p(\gamma+1-d)}\d v.$$
Choosing now $1 \leq p < \frac{d}{d-\gamma-1}$, one sees that
$$\int_{B_w} |v-w|^{p(\gamma+1-d)}\d v=|\S| \int_0^1 \dfrac{\d\varrho}{\varrho^{1-d-p(1+\gamma-d)}}  $$
is finite and is independent of $w$. Thus, there exists $C=C(p) >0$ such that
$$\int_{A} k_{\gamma}(v,w) m^{-1}(v)\d v \leq C \left(m^{-\frac{1}{p}}(w)\nu(A)^{1/q}+\nu(A)\right) \qquad \forall w \in \mathbb{R}^3.$$
Since $p >1$, this proves that \eqref{DP11} holds true. Let us now prove \eqref{DP22}. One first notices that
\begin{equation}\label{DP22r}
\sup_{|w| \leq r/2}\dfrac{m(w)}{(1+|w|)^{\gamma}}\int_{|v| > r} k_{\gamma}(v,w)  m^{-1}(v)\d v  \longrightarrow 0 \quad \text{ as } \quad r \to \infty.
\end{equation}
Indeed, one notices that
$$k_{\gamma}(v,w)  \leq C\,|v-w|^{\gamma+1-d} \exp\left(\frac{1}{4}\left[|w|^2 -|v|^2\right]\right)$$
with $\gamma+1-d \leq -1$. Therefore, if $|w| \leq r/2$ and $|v| > r$, one gets
$$k_{\gamma}(v,w)  \leq C\left(\frac{r}{2}\right)^{\gamma+1-d}  \exp\left(-\tfrac{3}{16}|v|^2\right)$$
where we used that $|v| > 2|w|.$ Then, \eqref{DP22r} follows easily since, in both the considered cases,
$$\IR \exp\left(-\tfrac{3}{16}|v|^2\right)\d\nu(v) < \infty.$$
Consequently, to prove \eqref{DP22}, it is enough to show that
\begin{equation}\label{DP22rr}
\sup_{|w| > r/2}\dfrac{m(w)}{(1+ |w|)^{\gamma}}\int_{|v | > r} k_{\gamma}(v,w) m^{-1}(v)\d v  \longrightarrow 0 \quad \text{ as } \quad r \to \infty.\end{equation}
According to  Proposition \ref{BCL}, for both the considered weights, there exist $C_0 > 0$ and $\alpha > 0$ such that
$$\int_{\mathbb{R}^d} k_{\gamma}(v,w) m^{-1}(v)\d v \leq C_0 \left(1+|w|^{\gamma-\alpha}\right)m^{-1}(w)$$
with $\alpha=s$ if $m(v)=\exp(-a|v|^{s})$ $s \in (0,1]$ or $\alpha=2$ if $m(v)=(1+|v|^{-\beta})$, $\beta > 0.$
Therefore, for any $r >0$,
$$\dfrac{m(w)}{(1+|w|)^{\gamma}}\int_{|v| > r} k_{\gamma}(v,w) m^{-1}(v)\d v \leq C_0 \dfrac{1+|w|^{\gamma-\alpha}}{(1+|w|)^{\gamma}}$$
and \eqref{DP22rr} follows since  $ \sup_{|w| > r/2} \dfrac{1+|w|^{\gamma-\alpha}}{(1+|w|)^{\gamma}} \to 0$ as $r \to \infty.$ This achieves to prove that $\mathbf{K}_\gamma\::\:\Y \to \X$ is weakly-compact.\end{proof}
\begin{nb}\label{noweight} Notice that the above result \emph{does not hold true} without weight $m^{-1}$ (i.e. if $a=0$ or $\beta=0$). Indeed, it is shown in \cite[Proposition 3.22]{m2As} that
$$\mathbf{K}_{d-2}\::\:L^1(\R^d;(1+|v|)\d v) \to L^1(\R^d;\d v)$$
is not weakly compact. The same method shows that actually, $\mathbf{K}_\gamma$ is not weakly compact for $\gamma \in [0,d-2]$. Notice that, in the absence of weight, it is \eqref{DP22rr} which is violated while \eqref{DP11} and \eqref{DP22r} still hold true.
\end{nb}
\begin{nb} Notice also that almost all the above computations are still valid for $\gamma < 0.$ In particular, \eqref{DP22r} and \eqref{DP11} still hold true for $-d \leq \gamma < 0$. However, the final argument in our proof does not hold true for $\gamma < 0$, since clearly
$$\lim_{r \to \infty} \sup_{|w| > r/2} \dfrac{1+|w|^{\gamma-\alpha}}{(1+|w|)^{\gamma}} =\infty.$$
\end{nb}

\begin{cor} For any $\gamma \in (0,d-2]$ and any $b(\cdot) \in L^1(\S)$, let $K$ denote the collision operator associated to $B(v-\vb,\sigma)=|v-\vb|^{\gamma}b(\cos\theta)$ defined by \eqref{gain}. Then 
$$K\,R(\lambda,T)\::\;\X \to \X$$
is positive and weakly compact for any $\lambda > 0.$ 
\end{cor} 
\begin{proof} The fact that $KR(\lambda,T)$ is positive is obvious. Given $\lambda > 0$, the range of $R(\lambda,T)$ is $\Y$ and one has to prove that
$$K \::\:\Y \to \X \quad \text{ is weakly-compact}.$$
Let $(b_{n})_{n}$ be a sequence of angular kernel such that $b_{n} \in L^{\infty}(\S)$ and $\|b_{n}-b\|_{L^{1}(\S)} \to 0$ as $n \to \infty$ and let $(K_{n})_{n}$ be the sequence of collision operator associated to $(b_{n})_{n}$. According to Lemma \ref{lemme:app}, $\lim_{n}\|K_{n}-K\|_{\mathscr{B}(\Y,\X)}=0.$ In particular, if $K_{n}$ is weakly compact for any $n \in \mathbb{N}$, so is $K$. It suffices therefore to prove that $K\::\:\Y \to \X$ is weakly-compact whenever $b \in L^{\infty}(\S).$ Now, for $b \in L^{\infty}(\S)$, one clearly has
$$Kf(v) \leq \|b\|_{L^{\infty}(\S)}\mathbf{K}_{\gamma}f(v) \qquad \forall v \in \R^{d}, \quad \forall f \in \Y, f \geq 0,$$
i.e. $K$ is dominated by the operator $\|b\|_{L^{\infty}(\S)}\mathbf{K}_{\gamma}.$ By Dunford-Pettis Theorem, the weak compactness of $\mathbf{K}_{\gamma}$ given by Proposition \ref{prop:comp} ensures that of $K$.
\end{proof}

%Being weakly-compact from $\Y\to\X$, the operator $K$ is an integral operator and there exists $\kappa(v,w)$ such that
%$$Kf(v)=\int_{\R^d} \kappa(v,w) f(w)\d\nu(w)$$
%where we recall that $\d\nu(v)=m^{-1}(v)\d v$ and $\X=L^1(\R^d,\d\nu).$
%\begin{nb} In the particular case of variable hard-spheres $\mathbf{K}_\gamma$, one sees that
%$$\kappa(v,w)=k_\gamma(v,w)m(w), \qquad \forall v,w \in \R^d \times\R^d.$$
%\end{nb}

From now on, we shall always consider that the angular kernel $b \in L^{1}(\S)$.

According to the previous Proposition, Theorem \ref{theommk} applies to this case yielding the following
\begin{theo}\label{theo:generation} If the weight function is given by \eqref{expwe} or \eqref{algwe}, then, for any $\gamma \in [0,d-2]$, the linear Boltzmann operator $(\L,\D(\L))$ with $\D(\L)=\Y $ is the generator of a positive \com $(V(t))_{t \geq 0}$ on $\X$ with moreover
$$\mathfrak{S}_{\mathrm{ess}}(V(t))=\mathfrak{S}_{\mathrm{ess}}(U(t)) \qquad \forall t \geq 0.$$
\end{theo}
\begin{nb}\label{spectrL} Notice that the spectrum $\mathfrak{S}(T)$ of $T$ in $\X$ is given by
$$\mathfrak{S}(T)=\mathrm{Range}(-\Sigma)=(-\infty,-\eta]$$
where $\eta=\mathrm{ess-inf}_{v \in \R^d} \Sigma(v) > 0$ is explicit. Moreover, the essential spectrum of $(\L,\D(\L))$ is then
$$\mathfrak{S}_{\mathrm{ess}}(\L)=\mathfrak{S}_{\mathrm{ess}}(T) \subset (-\infty,-\eta]$$
and $\mathfrak{S}(\L)=\mathfrak{S}_{\mathrm{ess}}(T) \bigcup \{\lambda_k\}_{k \geq 1}$ where $(\lambda_k)_k$ are eigenvalues of $(\L,\D(\L))$ with finite algebraic multiplicities.\end{nb}
\begin{nb}Notice that, since the kernel $k(\cdot,\cdot)$ is positive:
$$k(v,w) > 0 \qquad \text{ for a. e. } v,w \in \R^d \times \R^d $$
the semigroup $(V(t))_{t \geq 0}$ generated by $(\L,\D(\L))$ is \emph{irreducible} in $\X$. Actually,  if  $f \in \X$, $f \geq 0$, $f \neq 0$, then $[V(t)f](v) > 0$ for a. e. $v \in \R^d$. Indeed, the semigroup being given by a Dyson-Phillips expansion series
$$V(t)f=\sum_{n=0}^{\infty}V_{n}(t)$$
where $V_{0}(t)=U(t)$ and $V_{n+1}(t)f=\int_{0}^{t}V_{n}(t-s)KU(s)f	\d s$ for any $f \in \X$. Since the kernel $k(\cdot,\cdot)$ is positive, it is easy to check that $[V_{n}(t)f](v) > 0$ for a. e. $v \in \R^{d}$ if $f \in \X$, $f \geq 0$ non identically equal to zero. 
\end{nb}

Notice that $0$ is a simple eigenvalue of $\L$ with eigen-space spanned by $\M$. In particular, the spectral bound $s(\L)$ is such that $s(\L) \geq 0$. Finally, using Proposition \ref{propo:sLomega} one has
\begin{propo}\label{prop:type} The type $\omega_0(V)$ of the \com   $(V(t))_{t \geq 0}$ is zero, i.e.
$$\omega_0(V)=s(\L)=0.$$
 \end{propo}
\begin{proof} Let $\langle\cdot;\cdot\rangle_{\X,\X^\star}$ denotes the duality bracket between $\X$ and the dual space $\X^\star$. Now, since $\omega_{\mathrm{ess}}(V) < 0 \leq s(\L)$, one knows (see Remark \ref{spectrL}) that $s(\L)$ is an isolated eigenvalue of $\L$ with finite algebraic multiplicity. According to Proposition \ref{aren}, since the \com $(V(t))_{t \geq 0}$ is irreducible, $s(\L)$ is also an eigenvalue of the adjoint operator $\L^\star$ in $\X^\star$ associated to a \emph{positive eigenfunction} $g^\star$, $g^\star(v) > 0$ for a.e. $v \in \R^d$. Then, since $\L(\M)=0$, one has
$$0=\left\langle\L(\M)\;;\; g^\star\right\rangle_{\X,\X^\star}=\left\langle\M\;;\;\L^\star(g^\star)\right\rangle_{\X,\X^\star}=s(\L) \left\langle\M\;;\;g^\star\right\rangle_{\X,\X^\star}$$
and, since both $g^\star$ and $\M$ are positive, one has $\left\langle\M\;;\;g^\star\right\rangle_{\X,\X^\star} > 0$ from which necessarily $s(\L)=0.$ The conclusion follows from the well-know fact that, being the \com $(V(t))_{t \geq 0}$ positive, one has $\omega_0(V)=s(\L)$.
\end{proof}
\begin{nb} In the above proof, we used the fact that $s(\L)$ is an eigenvalue of the adjoint operator $(\L^\star,\D(\L^\star))$ associated to a positive eigenfunction $g^\star$. Actually, thanks to the conservative properties of $\L$ one can now be more precise.  Namely, one notices that
 \begin{equation}\label{mass}\int_{\R^d}\L f(v) \d v=0 \qquad \forall f \in \D(\L).\end{equation}
In particular, with the notations of the above proof, one deduces that
$$\left\langle \L f\;;\;m\right\rangle_{\X,\X^\star}=\left\langle f\;;\;\L^\star m\right\rangle_{\X,\X^\star}=\int_{\R^d} \L f(v) \d v=0 \qquad \forall f \in \D(\L).$$
In particular, $\L^\star\,m=0$ and $s(\L)=0$ is an eigenvalue of the adjoint operator $(\L^\star,\D(\L^\star))$  associated to the eigenfunction $g^\star=m$. Moreover, using again Proposition \ref{aren}, the spectral projection $\Pi_0$ associated to the eigenvalue $s(\L)=0$
of $\L$ is given by $\Pi_0=m \otimes \M$, i.e.
\begin{equation}\label{pi_0}
 \Pi_0 f=\langle f,m\rangle_{\X,\X^\star} \M =\varrho_f \M\;;\;\quad \text{ with } \varrho_f =\int_{\R^d} f(v)\d v\;;\quad \forall f \in \X.
\end{equation}
\end{nb}
We are in position to prove Theorem \ref{theo:introHard}

\begin{proof}[Proof of Theorem \ref{theo:introHard}] The fact that $\L$ is the generator of a positive semigroup in $\X$ has been already proven in Theorem \ref{theo:generation} so we focus here only on the large time behaviour of $(V(t))_{t\geq 0}.$ Combining Theorem \ref{theo:generation} with Proposition \ref{prop:type}, one has
$$\omega_{\mathrm{ess}}(V)=\omega_{\mathrm{ess}}(U) \leq -\eta < 0=\omega_{0}(V)=s(\L).$$
We get therefore the conclusion using Proposition \ref{prop:positive},  Theorem \ref{theo:engel}  together with the expression of $\Pi_0$ given by \eqref{pi_0}.
\end{proof}

\section{Convergence to equilibrium for linear BE with soft potentials}\label{sec:soft}

We investigate in this section the case of soft potentials. Our analysis is performed in the ``unweighted'' space
$$\X=L^{1}(\R^{d},\d v)$$ 
since we shall apply results pertaining to the ``substochastic theory'' of semigroups. We recall here that, in such a case, the collision frequency is bounded by virtue of \eqref{sigma}. We use the notations of the previous section. Using again the decomposition 
$$\L=K-T$$
with $T f(v)=\Sigma(v)f(v)$, one has $\D(T)=\X$ and, consequently, the $C_{0}$-semigroup  $(U(t))_{t \geq 0}$ in $\X$ defined by
$$U(t)f(v)=\exp(-\Sigma(v) t)f(v)\qquad t\geq 0,\:\:f \in \X,$$
is \emph{uniformly continuous} in the following sense:
\begin{equation}\label{normcontinuous}\lim_{t \to 0^{+}}\|U(t)-\mathbf{I}\|_{\mathscr{L}(\X)}=0\end{equation}
where $\mathscr{L}(X)$ stands for the Banach space  of all bounded operators in $\X$ endowed with the uniform operator topology induced by $\|\cdot\|_{\mathscr{L}(\X)}$ whereas $\mathbf{I}$ is the identity operator in $\X.$ Indeed, one checks easily that, for any $f \in \X$:
$$\|U(t)f-f\|_{\X} \leq \int_{\R^{d}}\left|\exp(-t\Sigma(v))-1\right|\,|f(v)|\d\nu(v) \leq t\int_{\R^{d}}\Sigma(v)\,|f(v)|\d\nu(v) 
\leq \sigma_{2}\,t \|f\|_{\X},$$
from which
$$\|U(t)-\mathbf{I}\|_{\mathscr{L}(\X)} \leq t\sigma_{2}\qquad \forall t > 0$$
where $\sigma_{2} >0$ is the positive constant appearing in \eqref{sigma}. Then, \eqref{normcontinuous} follows. Concerning the collision operator $K$, we have the following:
\begin{propo}\label{propo:bounded} Assume $m$ to be given by \eqref{expwe} or \eqref{algwe}. If $\gamma \in (-d,0]$ and $b \in L^{1}(\S)$ is nonnegative, then $K \::\:\X \to \X$ is bounded. 
\end{propo}
\begin{proof}
The proof uses some of the computations made in Lemma \ref{lemme:app}. Namely, given $f \in \X$ and $\gamma \in (-d,0]$, $b \in L^{1}(\S)$, one sees that 
\begin{multline*}
\|Kf\|_{\X} \leq \int_{\R^{2d}\times \S} B(v-\vb,\sigma) \,|f(v)|\,\M(\vb)\d v\d\vb\d\sigma\\
= \int_{\R^{2d}}b(\cos\theta) \,|v-\vb|^{\gamma}|f(v)|\,\M(\vb)\d v\d\vb\d\sigma.
\end{multline*}

Therefore, we get
$$\| Kf\|_{\X} \leq \| b\|_{L^{1}(\S)}\int_{\R^{2d}}\,|f(v)|\,\M(\vb)\,|v-\vb|^{\gamma}\,\d\vb\d v.$$
Since it clearly holds that $\M \in L^{q}(\R^{d})$ for any $q > 1$, one deduces from  Hardy-Littlewood-Sobolev inequality (see \cite[Theorem 4.3]{lieb}) that
$$\int_{\R^{2d}}\,|f(v)|\,\M(\vb)\,|v-\vb|^{\gamma}\,\d\vb\d v \leq C_{d}(\gamma)\,\|f\|_{L^{1}(\R^{d})}\,\|\M\|_{L^{q}(\R^{d})}$$
with some universal constant $C_{d}(\gamma)$ depending only on $\gamma$ and with $q >1$ such that $\gamma/d=1/q-1.$ Thish proves the result.
\end{proof}
With this in hands, one can prove a first non quantitative convergence theorem
\begin{propo}\label{propo:nonquantit} 
If  $\gamma \in (-d,0)$ and $b(\cdot) \in L^1(\S)$ then $\L$ generates a \com $(V(t))_{t \geq 0}$   in $\X$ which satisfies
$$\lim_{t \to \infty}\|V(t)f_{0}-\varrho_0 \M \|_\X =0$$
where $\varrho_0=\ds \int_{\R^d} f_{0}(v)\d v$ for any $f_{0} \in \X$.
\end{propo}

\begin{proof} Since both $T$ and $K$ are bounded operator, one has $\L \in \mathscr{L}(\X)$ and the  \com $(V(t))_{t \geq 0}$ generated by $\L$ is \emph{uniformly continuous}. Since the same holds for the \com $(U(t))_{t \geq 0}$, defining 
$$R_{1}(t)=V(t)-U(t)$$
it is clear that the mapping $t \geq 0 \mapsto R_{1}(t) \in \mathscr{L}(\X)$ is continuous (with respect to the operator norm). Since moreover $\mathrm{Ker}(\L) \neq \{0\}$ (since $\mathcal{M} \in \X$ and $\L(\mathcal{M})=0$) and the semigroup $(V(t))_{t \geq 0}$ is irreducible, one concludes with Theorem \ref{02law}.
\end{proof}

We can make the above result more precise and provide here a quantitative version of the above convergence result. Recall that the gain operator $K$ can be written as an integral operator
$$K f(v)=\int_{\R^{d}}k(v,w)f(w)\d w \qquad \forall f \in \X$$
for some nonnegative measurable kernel $k(v,w)$ such that
\begin{equation}\label{eq:sigK}
\Sigma(v)=\int_{\R^{d}}k(w,v)\d w \qquad \forall v \in \R^{d}.\end{equation}
Recalls that $\Sigma \in L^{\infty}(\R^{d})$ and $\inf_{v} \Sigma(v)=0$. The above identity will play a crucial role in the proof of the following:
\begin{propo}\label{prop:spectR}
Define the mapping ${\vartheta}\::\:\R^{+}\to \R^{+}$ by
\begin{equation}\label{Estimation resolvante}
\vartheta(r):=\frac{1}{r}\dfrac{1}{1-\frac{\Sigma _{\max }}{\sqrt{r^{2}+\Sigma _{\max
}^{2}}}}\ \ (r>0)   
\end{equation}%
where $\Sigma_{\max }:=\sup_{v}\Sigma(v) <\infty.$
Then, one has
$$i\,\R \setminus \{0\} \subset \varrho(\L)$$
and
$$\left\|R(i\alpha,\L)\right\|_{\mathscr{L}(\X)} \leq \vartheta\left(|\alpha|\right) \qquad \forall \alpha \in \R \setminus \{0\}.$$
\end{propo}
\begin{proof} Consider the resolvent equation
$$\lambda f-\L f=g, \qquad g \in \X, \:f \in \D(\L)=\X$$
where $\lambda \notin \mathrm{Range}(-\Sigma)=\mathfrak{S}(T)$ where $T$ is the multiplication operator by $-\Sigma$. Since $\L=K+T$ and $\lambda \in \varrho(T)$, such a problem is then equivalent to
$$\phi-KR(\lambda,T)\phi=g$$
where $\phi=(\lambda-T)f$. In particular, for such $\lambda$, one will get $\lambda \in \varrho (\L)$ if  
$$
\left\Vert  KR(\lambda,T)\right\Vert_{\mathscr{L}(\X)} <1 $$%
and then
\begin{equation}\label{eq:resol}
\phi=\sum_{n=0}^{\infty }\left( KR(\lambda,T)\right)
^{n}g \qquad \text{ and } \qquad f=\sum_{n=0}^{\infty}KR(\lambda,T)\left( KR(\lambda,T)\right)
^{n}g .\end{equation}
Let us now consider $\lambda=i\alpha$ with $\alpha \in \R$, $\alpha \neq 0.$ Then, $\alpha \notin \mathrm{Range}(-\Sigma)$ and
$R(i\alpha,T)$ is the multiplication operator by $\frac{1}{i\alpha+\Sigma}.$ Consequently,
$$KR(i\alpha,T)\phi=\int_{\R^{d}}
\dfrac{k(v,w)}{i\alpha +\Sigma (w)}\phi(w)\d w \qquad \forall \phi \in \X.$$
Consequently
\begin{equation*}\begin{split}
\|KR(i\alpha,T)\phi\|_{\X} &\leq \int_{\R^{d}}\d v\int_{\R^{d}}\frac{k(v,w)}{\sqrt{\alpha ^{2}+\Sigma ^{2}(w)}}%
\left\vert \phi(w)\right\vert \d w\\
&=\int_{\R^{d}}\frac{\Sigma(w)}{\sqrt{\alpha^{2}+\Sigma^{2}(w)}}|\phi(w)|\d w
\end{split}
\end{equation*}%
where we used Fubini's theorem and the identity \eqref{eq:sigK}. Therefore
$$\left\|KR(i\alpha,T)\right\|_{\mathscr{L}(\X)} \leq \sup_{w} \frac{\Sigma(w)}{\sqrt{\alpha^{2}+\Sigma^{2}(w)}} = \dfrac{\Sigma_{\max}}{\sqrt{\alpha^{2}+\Sigma_{\max}^{2}}} <1$$
since the mapping $r \geq 0 \mapsto \frac{r}{\sqrt{\alpha^{2}+r^{2}}}$ is nondecreasing. Then, $i\alpha \in \varrho(\L)$ and from \eqref{eq:resol}, we deduce that
$$
\|R(i\alpha,\L)\|_{\mathscr{L}(\X)} \leq 
\dfrac{1}{|\alpha|}\sum_{n=0}^{\infty }\left( \frac{\Sigma _{\max }}{\sqrt{\alpha
^{2}+\Sigma _{\max }^{2}}}\right) ^{n}  
= \frac{1}{\left\vert \alpha \right\vert }\frac{1}{1-\frac{\Sigma _{\max }}{%
\sqrt{\alpha ^{2}+\Sigma _{\max }^{2}}}}$$
where we used the fact that $\|R(i\alpha,T)\|_{\mathscr{L}(\X)}=\sup_{v}|i\alpha+\Sigma(v)|^{-1}=|\alpha|^{-1}.$ This proves the result.
\end{proof}

Introduce the mapping
$$\vartheta_{\log}(r):=\vartheta(r)\log \left( 1+\frac{\vartheta(r)}{r}\right) \qquad r > 0.$$
Notice that, since $\vartheta(\cdot)$ is strictly decreasing with
$$
\lim_{r\rightarrow +\infty }\vartheta(r)=0,\ \lim_{r\rightarrow 0_{+}}\vartheta(r)=+\infty . 
$$
the same holds true for $\vartheta_{\log}(\cdot)$ and $\vartheta_{\log}^{-1}\::\:\R^{+}\to \R^{+}$ denotes its inverse mapping.\\

Recall that the \com $(V(t))_{t \geq 0}$ generated by $\L$ in $\X$ is such that 
$$\lim_{t \to \infty}\|V(t)f_{0}-\mathbb{P}f_{0}\|_\X =0$$
where  the projection $\mathbb{P}$ is given by
$$\mathbb{P}f_{0}=\left(\int_{\R^d} f_{0}(v)\d v\right)\M \quad \text{ for any }  f_{0} \in \X.$$
Moreover, $\mathbb{P}$ is the spectral projection over $\mathrm{Ker}(\L)$ and, as such, the following holds
$$\X=\mathrm{Ker}(\L) \oplus \overline{\mathrm{Im}(\L)}.$$
Denote then by
$$\mathbf{Z}=\mathrm{Ker}(\L) \oplus  \mathrm{Im}(\L).$$
It is easy to check that $\mathbf{Z}$ is dense in $\X$ and one  deduces from the spectral properties of $\L$ established in Proposition \ref{prop:spectR}  the following
\begin{theo}\label{theo:rate}
Assume that $f \in  \mathbf{Z}$, then, for any $c \in (0,1)$,  
$$\|V(t)f-\mathbb{P}f\|_{\X}=\mathrm{O}\left(\vartheta_{\log}^{-1}(ct)\right) \qquad \text{ as } \quad t \to \infty.$$
\end{theo}
\begin{proof} 
Since its generator $\L$ is bounded, the \com $(V(t))_{t \geq 0}$ is analytic and therefore asymptotically analytic \cite{engel}. Then, combining Proposition \ref{prop:spectR} with \cite[Corollary 2.12]{Chill}, we deduce that, for any $c \in (0,1)$
\begin{equation}\label{vtL}
\|V(t)\L\,R(1,\L)\|_{\mathscr{L}(\X)}=\mathrm{O}(\vartheta_{\log}^{-1}(c\,t)) \qquad \text{ as } \quad t \to \infty.\end{equation}
In particular,
$$\|V(t)h\|= \mathrm{O}(\vartheta_{\log}^{-1}(c\,t)) \qquad \text{ as } \quad t \to \infty$$
for any $h \in \mathrm{Im}(\L)$ from which we get the result since $\mathbb{P}f=V(t)\mathbb{P}f$ for any $t\geq 0$.\end{proof}

One easily deduces the following algebraic rate of convergence:
\begin{cor}
Let $f \in \mathbf{Z}$. Then, for any $\varepsilon > 0$, there exists $C > 0$ such that
\begin{equation}\label{eq:alg}
\|V(t)f-\mathbb{P}f\|_{\X} \leq C\left(1+t\right)^{-\frac{1}{3+\varepsilon}} \qquad \forall t \geq 0.\end{equation}
\end{cor}
\begin{proof}
The proof consists simply in noticing that, since
$\vartheta(r) \simeq \frac{2\Sigma_{\max}^{2}}{r^{3}}$ whenever $r \simeq 0^{+}$,
for any $\varepsilon > 0$, there is some universal constant $\kappa_{\varepsilon} > 0$ such that
$$\vartheta_{\log}(r) \leq \dfrac{\kappa_{\varepsilon}}{r^{3+\varepsilon}} \qquad \text{ for } \quad r \simeq 0^{+}.$$
One deduces from that that
$$\vartheta_{\log}^{-1}(ct)=\mathrm{O}(t^{-\frac{1}{3+\varepsilon}}) \qquad \text{ for } \quad t \to \infty$$
and the result follows. 
\end{proof}

\begin{nb} 
Notice that the above convergence result, combined with Datko-Pazy's Theorem shows in particular that $\mathrm{Im}(\L)$ is not closed in $\X$. Indeed, if it were closed, we would have $\mathbf{Z}=\X$ and the above convergence theorem would imply that for, say $p > 4$ 
$$\int_{0}^{\infty}\left\|V(t)f-\mathbb{P}f\right\|_{\X}^{p}\d t < \infty \qquad \forall f \in \X.$$ 
According to Datko-Pazy's Theorem \cite[Theorem V.1.8]{engel}, the semigroup $(V(t)(\mathbb{I-P}))_{t\geq 0}$ would then be asymptotically exponentially stable. In particular, the spectrum of its generator $\L(\mathbb{I-P})$ would be contained in $\{z \in \mathbb{C}\,;\,\mathrm{Re}z \leq -\lambda_{*}\}$ for some $\lambda_{*} > 0$ and $0$ would be an isolated eigenvalue of $\L$. We know that this cannot be the case so $\mathbf{Z} \neq \X.$
\end{nb}

\begin{nb}\label{nb:caflisch} The study of the decay rate of the Boltzmann linearized operator in the hilbert space
$$\H=L^{2}(\R^{d},\M^{-1}(v)\d v)$$
has been performed by R. Caflisch in the seminal paper \cite{caflisch} in the dimension $d=3$ for $\gamma \in (-1,0).$ The same techniques can be performed for the linear operator $\L$, and, adopting the notations of the previous result, \cite[Theorem 3.1]{caflisch} would read as follows: for any $f \in \H$ such that 
\begin{equation}\label{eq:cafl}
\|f-\varrho_{f}\M\|_{\infty,a}=\sup_{v \in \R^{d}} \exp(a\,|v|^{2})\,|f(v)-\varrho_{f}\M(v)| < \infty \qquad \text{ for some } 0 < {a} < \frac{1}{4}\end{equation}
there exists $C >0 $ and $\lambda > 0$ (explicit and depending on $\gamma$, $\|f\|_{\alpha}$) such that
$$\|V(t)f-\varrho_{f}\M\|_{\H} \leq C\,\|f-\varrho_{f}\M\|_{\infty,a}\,\exp\left(-\lambda t^{\frac{2}{2-\gamma}}\right) \qquad \forall t \geq 0.$$
Notice that, here again, $\H$ splits as $\H=\mathrm{Ker}(\L) \oplus \overline{\mathrm{Im}(\L)}$. One can get rid of the very restrictive pointwise decay \eqref{eq:cafl} using an approach similar to the one we performed for the proof of Theorem \ref{theo:rate}. Indeed, recall that $\L$ is self-adjoint in the space $\H$ \cite{caflisch}) so that, classical spectral analysis of self-adjoint operators \cite{engel} ensures that
$$\|R(i\alpha\,,\,\L)\|_{\mathscr{L}(\H)} \leq \frac{1}{|\alpha|} \qquad \forall \alpha \in \R\setminus \{0\}.$$
Then, according to \cite[Theorem 1.5]{battychill},  such an estimate is equivalent to
$$\|V(t)\L\,R(1,\L)\|_{\mathscr{L}(\H)}=\mathrm{O}\left(\frac{1}{t}\right) \qquad \text{ as } t \to \infty.$$
Arguing then as in the proof of Theorem \ref{theo:rate}, one conclude that, if $f \in \mathrm{Ker}(\L) \oplus \mathrm{Im}(\L)$, then there exists $C>0$ such that
$$\|V(t)f-\varrho_{f}\,\M\|_{\H} \leq \frac{C}{1+t} \qquad \forall t \geq 0.$$
Again, as in Remark 4.6,  $\mathrm{Im}(\L)$ is not closed.
\end{nb}
\appendix

\section{Computation of the kernel $k(v,w)$}\label{sec:appA}

This section is devoted to the computation of the gain part operator $\mathbf{K}_\gamma$ for the special case in which the collision kernel is given by
$$B(\v-\vb,\sigma)=|\v-\vb|^\gamma \qquad \text{ for } \gamma \in (-d,\infty)$$
and, in particular, does not depend on the angular cross-section. In such a case, 
$$\mathbf{K}_\gamma=\int_{\R^d \times \S} |\v-\vb|^\gamma f(v')\M(\vb')\d \vb \d\sigma.$$
We aim to prove that $\mathbf{K}_\gamma$ admits the integral representation
$$\mathbf{K}_\gamma f(v)=\int_{\R^d} k_\gamma(v,w)f(w)\d w.$$
and shall resort to the so-called Carleman representation to compute the kernel $k_\gamma(v,w)$. Namely, one considers an alternative parametrization of post-collisional velocities as
$$v'=v- \langle \v-\vb,\n\rangle \n \qquad \text{ and } \quad \vb'=\vb + \langle \v-\vb,\n\rangle \n, \qquad \n \in \S$$
for which
$$\mathbf{K}_\gamma f(v)=\int_{\R^d \times \S} \widetilde{B}(v-\vb,\n) f(v')\M(\vb')\d\vb \d \n$$
where
$$\widetilde{B}(v-\vb,\n)=2^{d-2}\,\left|\left\langle \dfrac{v-\vb}{|v-\vb|}; \n\right\rangle\right|^{d-2} B(v-\vb,\sigma).$$
We refer to \cite[Chapter 2A. 4]{villani} for details. Since $|v'-\v|=|\langle v-\vb,\n\rangle|$, one gets that $\widetilde{B}(\v-\vb,\n)=2^{d-2} \left|v'-v\right|^{d-2} \,\left|v-\vb\right|^{\gamma-(d-2)}$ and
$$\mathbf{K}_\gamma f(v)=2^{d-2} \int_{\R^d \times \S} \left|v-\vb\right|^{\gamma-(d-2)}\left|v'-v\right|^{d-2} f(v')\M(\vb')\d \vb \d\n.$$
For fixed $\n \in \S$, we perform the change of variables $\vb \mapsto u=\vb-\v$ to get
$$\mathbf{K}_\gamma f(v)=2^{d-2} \int_{\R^d \times \S} f(v+\langle u,\n\rangle \n) \M(\v+u-\langle u,\n\rangle \n)|u|^{\gamma-(d-2)}\,|\langle u, \n\rangle|^{d-2}\d u \d\n$$
and then, setting $u=r\n + z$, $z \perp \n$, $r \in \R$, we get
$$\mathbf{K}_\gamma f(v)=2^{d-2} \int_{\R \times \S} f(v+ r\n)|r|^{d-2}\d r \d\n\left(\int_{\n^\perp} \M(\v +z)\,|z+r\n|^{\gamma-(d-2)} \d z\right).$$
Setting now $w=v+r\n$, $\d w=2|r|^{1-d}\d r\d\n$, we get
$$\mathbf{K}_\gamma f(v)=2^{d-1} \int_{\R^d} f(w) \dfrac{\d w}{|v-w|} \left(\int_{(v-w)^\perp} \M(v + z)\,|z+w-v|^{\gamma-(d-2)} \d z\right)$$
i.e.
$$\mathbf{K}_\gamma f(v)= \int_{\R^d} k_\gamma(v,w)\,f(w)\d w$$
with
$$k_\gamma(v,w)=2^{d-1}\,|v-w|^{-1}\int_{(v-w)^{\perp}}  \M(v+z) \,|z+w-v|^{\gamma-(d-2)}\d z.$$
Let us compute the above kernel in a more precise way. Let $v,w \in \R^d$ be given. We denote the center of mass by $V=\dfrac{v+w}{2}$. Then, for any $z \in (v-w)^\perp$, one has
$$|v+z|^2=|V + z|^2 + \frac{1}{4}|v-w|^2 + \frac{1}{2}\left(|v|^2-|w|^2\right)$$
and, setting $V=V_0+V_\perp$ where $V_\perp \perp v-w$ and $V_0$ parallel to $(v-w)$ we get
$$|V + z|^2 =|V_0|^2 + |V_\perp + z|^2 \qquad \text{ with } \quad |V_0|^2=\dfrac{\left(|v|^2-|w|^2\right)^2}{4|v-w|^2}$$
i.e.
$$|v+z|^2 = |V_\perp + z|^2 + \dfrac{1}{4}\left[|v-w| + \dfrac{|v|^2-|w|^2}{|v-w|}\right]^2.$$
Therefore
\begin{multline*}k_\gamma(v,w)=2^{d-1}(2\pi)^{-d/2}\,|v-w|^{-1}\exp\left(-\frac{1}{8}\left[|v-w| + \dfrac{|v|^2-|w|^2}{|v-w|}\right]^2\right)\\
\int_{(v-w)^\perp} \exp\left(-\frac{|V_\perp + z|^2}{2}\right)\,\left|z +w-v \right|^{\gamma-(d-2)}\d z.\end{multline*}
We denote by $I_{\gamma}(v,w)$ this last integral:
\begin{equation*}\begin{split}
I_{\gamma}(v,w)&=\int_{(v-w)^\perp} \exp\left(-\frac{|V_\perp + z|^2}{2}\right)\,\left|z +w-v \right|^{\gamma-(d-2)}\d z\\
&=\int_{(v-w)^\perp} \exp\left(-\frac{|\xi|^2}{2}\right)\,\left|V_{\perp}-\xi +v-w \right|^{\gamma-(d-2)}\d \xi
\end{split}\end{equation*}
where we recall that $V_\perp$ is orthogonal to $v-w$. 
One sees in particular that, if $\gamma=d-2$, the last integral is constant (independent of $v,w$): 
\begin{equation*}
I_{d-2}(v,w)=\int_{\R^{d-1}} \exp\left(-\frac{|x|^2}{2}\right)\d x=(2\pi)^{\frac{d-1}{2}},\end{equation*}
so that
\begin{equation}\label{kd-2}
k_{d-2}(v,w)=2^{d-1}(2\pi)^{-1/2}\,|v-w|^{-1}\exp\left(-\frac{1}{8}\left[|v-w| + \dfrac{|v|^2-|w|^2}{|v-w|}\right]^2\right) \qquad \forall v,w .
\end{equation}
For general $\gamma \leq d-2$, one notices that, being  both $V_{\perp}$ and $\xi$ orthogonal to $v-w$:
$$I_{\gamma}(v,w)=\int_{(v-w)^{\perp}} \exp\left(-\frac{|\xi|^{2}}{2}\right)\left(|V_{\perp}-\xi|^{2}+|v-w|^{2}\right)^{\frac{\gamma-d+2}{2}}\d\xi$$
so that
\begin{equation}\label{compg}I_{\gamma}(v,w)\leq |v-w|^{\gamma-(d-2)}I_{d-2}(v,w) \qquad \forall v,w\,\:\forall \gamma \leq d-2\end{equation}
and
\begin{equation}\label{k-d-2}k_\gamma(v,w) \leq |v-w|^{\gamma-(d-2)}k_{d-2}(v,w) \qquad \forall \gamma \leq d-2.\end{equation}
Such a rough estimate is enough to prove a general property of $k_{\gamma}$ (already established in \cite[Proposition A.1]{BCL} in the case of hard-spheres in dimension $d=3$):
\begin{propo}\label{BCL} The following holds.
\begin{enumerate}
\item Assume $m(v)=\exp(-a|v|^s)$, $a >0$, $s \in (0,1]$ and set
$$H_\gamma(w)=\int_{\R^d} k_{\gamma}(v,w)m^{-1}(v)\d v, \qquad \qquad w \in \R^d, \quad \gamma \leq d-2.$$
Then, there exists a positive constant $C_0  >0$ such that
$$H_\gamma(w) \leq C_0(1 + |w|^{\gamma-s})\,m^{-1}(w) \qquad \forall w \in \R^d.$$
\item If $m(v)=\left(1+|v|^{\beta}\right)^{-1}$ for $\beta > 0$ and let $H_{\gamma}$ be defined as above. Then, there exists a positive constant $C_{1} > 0$ such that
$$H_\gamma(w) \leq C_1(1 + |w|^{\gamma-2})\,m^{-1}(w) \qquad \forall w \in \R^d.$$
\end{enumerate}
\end{propo}
\begin{proof} In both cases, recall that, according to \eqref{k-d-2}, one has
$$H_\gamma(w) \leq  \int_{\R^d} |v-w |^{\gamma-(d-2)}k_{d-2}(v,w)m^{-1}(v)\d v$$
where
 $k_{d-2}(v,w)$ is given by \eqref{kd-2}. Taking into account that
$\frac{|v|^2-{|w|}^2}{|v-w|} - |v-w|
  =
  2 \frac{v-w}{|v-w|}\cdot w$ we may rewrite \eqref{kd-2} as
\begin{equation}
  \label{eq:k2}
  k_{d-2}(v,w)
  = 2^{d-1}(2\pi)^{-1/2} |v-w|^{-1} \exp \left\{
    -\frac{1}{2}
    \left(
 |v-w| +  \frac{v-w}{|v-w|} \cdot w
    \right)^2
  \right\}.
\end{equation}
Assume now that $m$ is an exponential weight, i.e. $m(v)=\exp(-a|v|^{s})$ for $a > 0$ and $s \in (0,1].$ Performing the change of variables $u=v-w$ and using spherical coordinates (with $\varrho=|u|$ and $\varrho|w|\cos \theta=u \cdot w$)  one gets therefore
\begin{equation*}
H_\gamma(w) \leq C_d |\mathbb{S}^{d-2}| \int_{0}^\infty \d\varrho \int_0^\pi G(\varrho, \theta) \d\theta
\end{equation*}
for some universal constant $C_d > 0$ and where
$$G(\varrho,\theta)=\varrho^{\gamma} \exp \left\{
      -\frac{1}{2}
      \bigg(\varrho +  |w|\cos\theta
      \bigg)^2
      +  a       \bigg(\varrho^2+|w|^2 + 2\varrho|w|\cos \theta \bigg)^{s/2}\right\}\left(\sin\theta\right)^{d-2}.$$
Setting now  $A=[0,\infty) \times [0,\pi]$, we  split $A$ into the two regions of integration:
$$A_1= \{(\varrho,\theta) \in A \, ;\,3 |w| \cos\theta \geq
    - 2 \varrho \} \qquad \text{ and } \qquad A_2=A\setminus A_1.$$
Notice first that, since $\cos\theta \leq 1$ and $s\in (0,1]$
$$\exp\left(a(\varrho^2+|w|^2+2\varrho|w|\cos\theta)^{s/2}\right) \leq \exp\left(a(\varrho+|w|)^{s}\right) \leq \exp(a \varrho^s)\exp(a |w|^s) \qquad \forall (\varrho,y) \in A.$$
Moreover, since $ \varrho +  |w|\cos\theta \geq \varrho/3$ for any $(\varrho,\theta) \in A_1$ we have
  \begin{equation}\begin{split}
    \label{eq:region}
    \int_{A_1} G(\varrho,\theta)   \d\varrho\d \theta
    &\leq \exp(a |w|^s) \int_0^\infty \varrho^{\gamma}\d\varrho \! \int_{0}^\pi  \exp
    \left( -\frac{1}{18}\varrho^2 + a \varrho^s \right) \,(\sin\theta)^{d-2}\,\d \theta
    \\
  &\leq C_1 \exp(a |w|^s) = C_1 m^{-1}(w) \end{split}\end{equation}
  since the integral is convergent. Let us estimate now the integral over $A_2$.  For any $(\varrho,\theta) \in A_2$, one notices first that
$$\varrho^2 + |w|^2 + 2 \varrho|w|\cos \theta < |w|^2 -\varrho^2/3 \qquad \text{ and } \qquad \varrho \leq (3/2) |w|,$$
so that
  \begin{multline}
    \label{eq:region2}
    \int_{A_2}  G(\varrho,\theta) \d\varrho\d \theta
    \\
    \leq \int_0^{(3/2)|w|} \varrho^{\gamma}\exp \left( a \left(|w|^2 -\varrho^2/3
      \right)^{s/2} \right) \d\varrho\int_{0}^\pi \exp \left( -\frac{1}{2} \left( \varrho +   |w|\cos\theta \right)^2 \right) \,\left(\sin\theta\right)^{ d-2}\,\d \theta
  \end{multline}
  To carry out the $\theta$-integral, perform the change of variables  $z =
   \varrho + |w|\cos \theta$ to get
  \begin{equation*}
    \int_{0}^\pi
    \exp \left(  -\frac{1}{2} \left( \varrho +   |w|\cos\theta \right)^2
    \right)
     \,\left(\sin \theta\right)^{d-2}\,\d \theta
    \leq
    \frac{1}{|w|}
    \int_{-\infty}^{\infty}
    \exp \left( - z^2/2
    \right)
    \,\d z
    =
    \frac{C_2}{|w|}
  \end{equation*}
  for some explicit $C_2 >0$ \footnote{notice that this proof needs  $d \geq 3$}. Plugging this  in \eqref{eq:region2} we obtain
$$\int_{A_2}  G(\varrho, \theta)  \d\varrho\d \theta
    \leq \frac{C_2}{|w|}\int_0^{(3/2)|w|} \varrho^{\gamma}\exp \left( a \left(|w|^2 -\varrho^2/3
      \right)^{s/2} \right) \d\varrho.$$
Setting now  $x=|w|^2 - \varrho^2/3$, we obtain
$$\int_{A_2}  G(\varrho,\cos \theta)  \d\varrho\d \theta\leq \dfrac{3C_2}{2\,|w|}3^{\frac{\gamma-1}{2}}\int_{|w|^2/4}^{|w|^2}\left(|w|^2-x\right)^{\frac{\gamma-1}{2}}\exp(a x^{s/2})\d x.$$
Since $\tfrac{3}{4}|w|^2 \leq |w|^2-x \leq |w|^2$ for any $x \in (\tfrac{1}{4}|w|^2,|w|^2)$, one sees that
\begin{equation}\label{eq:region2-2}\int_{A_2}F(\varrho,y)\d\varrho\d y \leq C_3|w|^{\gamma-2}\int_{0}^{|w|^2}\exp(a x^{s/2})\d x\end{equation}
for some explicit $C_3 > 0$ depending only on $\gamma$. We observe now that, for any $r >0$, \begin{equation*}\begin{split}
    \int_0^r \exp ( a x^{s/2} ) \,\d x
    &\leq \frac{2}{as}
    r^{1-s/2} \int_0^r \frac{as}{2} x^{s/2-1}\exp ( a x^{s/2}
    ) \,\d x
    \\
    &= \frac{2}{as} r^{1-s/2} \int_0^r \frac{\d}{\d x}\exp  ( a
      x^{s/2} ) \,\d x \leq \frac{2}{as} r^{1-s/2} \exp ( a
      r^{s/2} ).\end{split}\end{equation*}
 Using this in \eqref{eq:region2-2} for $r = |w|^2$ we get
\begin{equation}
    \label{eq:region2-3}
    \int_{A_2}F(\varrho,y)\d\varrho\d y
    \leq
    \frac{2C_3}{a\,s}
    |w|^{\gamma-s}
    \exp \left( a |w|^{s} \right).
  \end{equation}
    Putting together \eqref{eq:region} and \eqref{eq:region2-3} we finally obtain the result.
    
  Assume now that $m^{-1}(v)=1+|v|^{\beta}$ for $\beta >0.$ One argues in a similar way. Namely, as above,  
  \begin{equation*}
H_\gamma(w) \leq C_d |\mathbb{S}^{d-2}| \int_{A} G(\varrho, \theta) \d\varrho\d\theta
\end{equation*}
for some universal constant $C_d > 0$ and where now
$$G(\varrho,\theta)=\varrho^{\gamma} \exp \left\{
      -\frac{1}{2}
      \bigg(\varrho +  |w|\cos\theta
      \bigg)^2\right\}
       \left(1+       \bigg(\varrho^2+|w|^2 + 2\varrho|w|\cos \theta \bigg)^{\beta/2}\right)\left(\sin\theta\right)^{d-2}.$$
As above, one splits the region $A=[0,\infty) \times [0,\pi]$  into the two above regions $A_{1}$ and $A_{2}$. One clearly has
$$  \left(1+       \bigg(\varrho^2+|w|^2 + 2\varrho|w|\cos \theta \bigg)^{\beta/2}\right)\leq (1+2^{\beta-1})\left(1+\varrho^{\beta}\right)\left(1+|w|^{\beta}\right)$$
from which, as above,
$$\int_{A_{1}}G(\varrho,\theta)\d\varrho\d\theta \leq \left(1+2^{\beta-1}\right)m^{-1}(w)\int_{A_{1}}\varrho^{\gamma}\exp\left(-\frac{1}{18}\varrho^{2}\right)\left(1+\varrho^{\beta}\right)\left(\sin\theta\right)^{d-2}\d\varrho\d\theta$$
so that there is $C_{1} > 0$ such that
\begin{equation}\label{eq:reg1}
\int_{A_{1}}G(\varrho,\theta)\d\varrho\d\theta \leq C_{1}m^{-1}(w).\end{equation}
Arguing now as above, one gets that
\begin{multline*}
\int_{A_{2}} G(\varrho,\theta)\d\varrho\d\theta \leq \int_{0}^{\frac{3}{2}|w|^{2}}\varrho^{\gamma}\left(1+\left(|w|^{2}-\frac{\varrho^{2}}{3}\right)^{\frac{\beta}{2}}\right)\d\varrho\\
\int_{0}^{\pi}\exp\left(-\frac{1}{2}\left(\varrho+|w|\cos\theta\right)^{2}\right) \left(\sin\theta\right)^{d-2}\d\theta\\
\leq \frac{C}{|w|}\int_{0}^{\frac{3}{2}|w|}\varrho^{\gamma}\,\left(1+\left(|w|^{2}-\frac{\varrho^{2}}{3}\right)^{\frac{\beta}{2}}\right)\d\varrho\end{multline*}
for some positive constant $C >0$. One computes the above integral as in the previous part and obtain that there exists a positive constant $C_{2} > 0$ such that
$$\int_{0}^{\frac{3}{2}|w|}\varrho^{\gamma}\,\left(1+\left(|w|^{2}-\frac{\varrho^{2}}{3}\right)^{\frac{\beta}{2}}\right)\d\varrho \leq C_{2}\,|w|^{\gamma-1}\int_{\frac{|w|^{2}}{4}}^{|w|^{2}}\left(1+x^{\beta/2}\right)\d x \leq C_{2}|w|^{\gamma-1}m^{-1}(w).$$
Gathering all together, we get that there exists $C_{3} > 0$ such that
\begin{equation}\label{eq:reg2}
\int_{A_{2}}G(\varrho,\theta)\d\varrho\d\theta \leq C_{3}\,|w|^{\gamma-2}\,m^{-1}(w).\end{equation}
Combining \eqref{eq:reg1} and \eqref{eq:reg2} we get the desired conclusion.
    \end{proof}
     
\begin{nb}
Notice that the above computations are valid for any $\gamma \in (-d,\infty)$ (i.e. it is valid for both hard and soft potentials). However, it turns out to be particularly relevant for hard potentials $\gamma \in [0,d-2]$. For soft potentials, it will become inoperative as explained by the following reasoning. Recalling that $\X=L^1(\R^d,\d\nu(v))$ with $\d\nu(v)=m^{-1}(v)\d v$, one can write $\mathbf{K}_\gamma$ as an integral operator \emph{for the measure $\d\nu$}
$$\mathbf{K}_\gamma f(v)=\int_{\R^d}\mathbf{k}_\gamma(v,w) f(w)\d\nu(w) $$
by simply setting
$$\mathbf{k}_\gamma(v,w)=k_\gamma(v,w)m(w), \qquad \forall v,w \in \R^d.$$
In such a case, the above Proposition asserts that
$$\int_{\R^d}\mathbf{k}_\gamma(v,w)\d\nu(v)=m(w)H_\gamma(w) \leq C_0\left(1+|w|^{\gamma-\alpha}\right)$$
where $\alpha=s$ if $m(v)=\exp(-a|v|^{s})$, $s \in (0,1]$ and $\alpha=2$ if $m(v)=(1+|v|^{\beta})^{-1}$ $(\beta > 0).$ 
Since, for $\gamma \geq 0$ the mapping
$$w \mapsto \frac{1+|w|^{\gamma-\alpha}}{(1+|w|)^\gamma} \quad \text{ is bounded}$$
we see that, for \emph{hard potentials} there exists $\vartheta \in L^\infty(\R^d)$ such that
$$\int_{\R^d}\mathbf{k}_\gamma(v,w)\d\nu(v) \leq\vartheta(w)\,\Sigma(w) \qquad \text{ for a. e. } w \in \R^d$$
i.e. \cite[Inequality (5)]{MMK} is satisfied. Notice however that nothing guarantees that the subcriticality condition \cite[Inequality (6)]{MMK} is met in $\X$. It it not clear if the above is still true for soft potential, i.e. whenever $\gamma < 0$.
 \end{nb}
    
We go on with properties of the $k_{\gamma}(v,w)$ for positive $\gamma$. First of all, the detailed balance law holds true:
$$k_{\gamma}(v,w)\M(w)=k_{\gamma}(w,v)\M(v) \qquad \forall v,w \in \R^d \times \R^d$$
This leads us to define
$$p_{\gamma}(v,w)=\M^{-1/2}(\v)k_{\gamma}(\v,w){\M}^{1/2}(w),\qquad \qquad \v,w \in \R^d \times \R^d,$$
so that $p_{\gamma}(v,w)=p_{\gamma}(w,v)$. Notice that, 
$$p_{\gamma}(v,w) \leq |v-w|^{\gamma-d+2}\M^{-1/2}k_{d-2}(\v,w)\M^{1/2}(w)$$
from which we easily see that
$$p_{\gamma}(v,w) \leq C\,|\v-w|^{\gamma-d+1}\exp\left\{-\frac{1}{8}\left(|v-w|^{2}+\frac{\left(|v|^{2}-|w|^{2}\right)^{2}}{|\v-w|^{2}}\right)\right\} \qquad \forall \v,w \in \R^{2d}.$$

This yields to the following whose prove is a direct adaptation of that given in \cite[Lemma 3.2]{arlotti}.
\begin{lemme}\label{lemmG}
For any $\gamma > \frac{d-2}{2}$, there exists $C >0$
such that
$$\int_{\R^d}|p_{\gamma}(v,w)|^2  \d w  \leq \dfrac{C}{(1+|v|)}, \qquad \forall v \in \R^d.$$
\end{lemme}
\begin{proof} We use the above bound to get that there exists $C >0$ such that
\begin{multline*}
\int_{\R^{d}}p_{\gamma}(v,w)^{2}\d w \leq C \int_{\R^{d}}|v-w|^{2\gamma+2-2d}\exp\left\{-\frac{1}{4}\left(|v-w|^{2}+\dfrac{\left(|v|^{2}-|w|^{2}\right)^{2}}{|v-w|^{2}}\right)\right\}\d w\\
\leq C \int_{\R^{d}}|u|^{2\gamma+2-2d}\exp\left\{-\frac{1}{4}\left(|u|^{2}+\dfrac{\left(|v|^{2}-|u-v|^{2}\right)^{2}}{|u|^{2}}\right)\right\}\d u
\end{multline*}
where we set $u=v-w$. Using spherical coordinates $u=\varrho\,\d\n$ and denoting by $\theta$ the angle between $u$ and $v$, one readily gets
\begin{multline*}
\int_{\R^{d}}p_{\gamma}(v,w)^{2}\d w \leq C \int_{0}^{\infty}\varrho^{2\gamma+1-d}\d\varrho\int_{\S}\exp\left(-\frac{1}{4}\left(\varrho^{2}+(2|v|\cos\theta-\varrho)^{2}\right)\right)\d\n\\
=C \int_{0}^{\infty}\varrho^{2\gamma+1-d}\d\varrho\int_{0}^{\pi}\sin^{d-2}(\theta)\exp\left(-\frac{1}{4}\left(\varrho^{2}+(2|v|\cos\theta-\varrho)^{2}\right)\right)\d\theta.\end{multline*}
Notice that the assumption $2\gamma+1-d > 0$ implies
$$\sup_{v \in \R^{d}}\int_{\R^{d}}p_{\gamma}(v,w)^{2}\d w \leq C \pi \int_{0}^{\infty}\varrho^{2\gamma+1-d} \exp\left(-\frac{\varrho^{2}}{4}\right)\d\varrho < \infty.$$
Set for simplicity
$$I(v)=\int_{0}^{\infty}\varrho^{2\gamma+1-d}\d\varrho\int_{0}^{\pi}\sin^{d-2}(\theta)\exp\left(-\frac{1}{4}\left(\varrho^{2}+(2|v|\cos\theta-\varrho)^{2}\right)\right)\d\theta, \qquad v \in \R^{d}.$$
As in \cite[Lemma 3.1]{arlotti}, perform the change of variables $x=\varrho/|v|-2\cos\theta$, $y=\varrho/|v|$, one gets 
\begin{multline*}
I(v)=\frac{|v|^{2\gamma+2-d}}{2}\int_{\Omega}y^{2\gamma+1-d}\left(1-\frac{(x-y)^{2}}{4}\right)^{\frac{d-3}{2}}\exp\left\{-\frac{|v|^{2}}{4}\left(x^{2}+y^{2}\right)\right\}\d x\d y\\
\leq \frac{|v|^{2\gamma+2-d}}{2}\int_{\Omega}y^{2\gamma+1-d}\exp\left\{-\frac{|v|^{2}}{4}\left(x^{2}+y^{2}\right)\right\}\d x\d y\end{multline*}
where $\Omega=\{(x,y) \in \R^{2}\,;\,y \geq 0\,;\,|x-y|\leq 2\}.$ Changing the variables again as $t=|v|x/2$ and $s=|v|y/2$ one finally gets
$$I(v) \leq \frac{2^{2\gamma+2-d}}{|v|}\int_{\R}\d t\int_{0}^{\infty}\exp\left\{-\left(t^{2}+s^{2}\right)\right\}\,s^{2\gamma+1-d}\d s$$
and using again that $2\gamma+1-d$ one sees that this last integral is convergent which gives the conclusion.
\end{proof}

\section{Quantitative estimates  of the spectral gap for hard-potentials}\label{sec:appB}

We present here a quantitative version of Theorem \ref{theo:introHard} which relies on recent advances on the so-called factorisation method and enlargement of the functional space, initiated in \cite{Mo} for the nonlinear Boltzmann equation and extended to an abstract framework in \cite{gualdani}. The method consists in using the fact that, in the hilbert space 
$$\mathcal{H}=L^{2}(\R^{d},\M^{-1}\d v)$$
the spectral gap of the linear Boltzmann operator $\L$ is explicit (at least for hard-spheres interactions \cite{LMT} but the approach given then can be extended easily to the general hard potential dealt with in this paper). In particular, since $\L$ is self-adjoint in this space, the decay of the associated semigroup is directly deduced from the spectral structure of the generator. 

To avoid confusion, in this section, we shall denote $\L_{2}$ the linear Boltzmann in the Hilbert space $\mathcal{H}$ while $\L_{1}$ will stand for the linear Boltzmann operator in $\X$ where $\X$ has been defined in the Introduction:
$$\X=L^{1}(\R^{d},m^{-1}(v)\d v)$$
where the weight $m$ is given by \eqref{expwe} or \eqref{algwe}. We also denote by $(V_{2}(t))_{t\geq 0}$ and $(V_{1}(t))_{t \geq 0}$ the $C_{0}$-semigroups generated respectively by $\L_{2}$ and $\L_{1}.$ 

Notice that we restrict ourselves, for simplicity, to the case in which $b \in L^{\infty}(\S)$, which, as seen in Section \ref{sec:hard} is no restriction. Then, again, there is no loss in generality in assuming that $b$ is constant, $b\equiv 1.$ In such a case, the linear operator $K$ corresponds simply to $\mathbf{K}_{\gamma}.$

We begin by recalling the following result (see \cite{LMT, levermore})
\begin{theo}
For $\gamma \in [0,d-2]$, the linear operator $(\L_{2},\D(\L_{2})$ is negative self-adjoint in $\mathcal{H}$ where
$$\D(\L_{2})=\left\{f \in \mathcal{H}\,;\,\Sigma f \in \mathcal{H}\right\}=L^{2}(\R^{d},\left(1+|v|^{\gamma}\right)\M^{-1}(v)\d v).$$
Moreover, $\mathrm{Ker}(\L_{2})=\mathrm{span}(\M)$ and there exists some \emph{explicit} $\mu_{2} > 0$ such that
$$-\int_{\R^{d}}f\,\L_{2}f\,\M^{-1}\d v \geq \mu_{2}\|f\|_{\mathcal{H}} \qquad \forall f \in \D(\L_{2}) \quad \text{ with } \quad \int_{\R^{d}}f(v)\d v=0.$$
As a consequence, 
$$\|V_{2}(t)f -\varrho_f \M \|_{\mathcal{H}} \leq  \exp(-\mu_2 t)\|f \|_{\mathcal{H}} \qquad \text{ for all   }  f \in \mathcal{H}, \quad \forall t \geq 0$$
where $\varrho_f=\ds \int_{\R^d} f(v)\d v$ for any $f \in \mathcal{H}$.
\end{theo}

Notice that, the spectrum of $\L_{2}$ in $\H$ can be expressed as follows
$$\mathfrak{S}(\L_{2})=\left(-\infty,-\eta\right] \cup  \left\{\lambda_{n}\right\}_{n}$$
where $\lambda_{1}=0 > \lambda_{2} >  \lambda_{n} >  \lambda_{n+1}$ $(n \in \mathbb{R})$ are nonpositive eigenvalues of $\L_{2}$ with finite algebraic multiplicities and $\lim_{n\to\infty}\lambda_{n}=-\eta=-\inf_{v \in \R^{d}}\Sigma(v) > 0.$ For any $n \in \mathbb{N}$, we denote by
$\Pi_{\L_{2},n}$ the spectral projection associated to $\lambda_{n}$. Notice that $\mu_{2}=-\lambda_{2}.$\medskip

In order to extend the above result to the weighted $L^{1}$-space $\X$, we shall resort to the following that we state here for the above operators $\L_{2}$ and $\L_{1}$:
\begin{theo}{\cite[Theorem 2.1]{gualdani}}\label{theoMMG}
Assume that the linear operator $(\L_{1},\D(L_{1}))$ in $\X$ admits the following splitting:
$$\L_{1}=\mathcal{A}_{1}+ \mathcal{B}_{1}$$
where
\begin{enumerate}[(i)]
\item $\mathcal{A}_{1}\::\:\X  \to \H$ is bounded;
\item the operator
  $\mathcal{B}_{1}\::\:\D(\mathcal{B}_{1}) \to \X$ (with
  $\D(\mathcal{B}_{1})=\D(\L_{1})=\Y$) is $\beta$-dissipative for some
  positive $\beta >0$, i.e.
  \begin{equation}
    \label{Bdiss}
    \IR \mathrm{sign}f(v) \mathcal{B}_{1}f(v)m^{-1}(v)\d v
    \leq -\beta  \|f\|_{\Y}
    \qquad \forall f \in \Y.
  \end{equation}
\end{enumerate}
Then, for any $\varepsilon > 0$, 
$$\mathfrak{S}(\L_{1}) \cap \{z \in \mathbb{C}\;\,\,\mathrm{Re}z \geq -\eta+\varepsilon\}=\mathfrak{S}(\L_{2}) \cap \{z \in \mathbb{C}\,;\,\,\mathrm{Re}z \geq -\eta +\varepsilon\}=\{\lambda_{1},\ldots, \lambda_{N_{\varepsilon}}\}$$
for some finite $N_{\varepsilon} \in \mathbb{N}.$ Moreover, for any $k \in \{1,\ldots,N_{\varepsilon}\}$, $ \lambda_{k}$ is an eigenvalue of $\L_{1}$ with algebraic multiplicity which is finite and equal to the one of $\lambda_{k}$ as an eigenvalue of $\L_{2}$ and the associated spectral projection $\Pi_{\L_1,k}$ is given by $\Pi_{\L_{1},k}\vert_{\H}=\Pi_{\L_{2},k}$.
\end{theo}

The above Theorem asserts that, provided $\L_{1}$ admits the above splitting (i)-(ii) in the weighted space $\L_{1}$, the second eigenvalue of $\L_{1}$ in $\X$ is $\lambda_{2}=-\mu_{2}$. This, combines with the results of Section \ref{sec:hard} would implies that the constant $\lambda_{*}$ appearing in Theorem \ref{theo:introHard} can be chosen as $\mu_{2}+\delta$ for \emph{any} $\delta >0$. Since $\mu_{2}$ is explicit, this provides a quantitative estimate of the speed of convergence in Theorem \ref{theo:introHard}. 
\begin{nb}
Notice that it would be possible to invoke also results from \cite{gualdani}, especially, \cite[Theorem 2.13]{gualdani}, to deduce decay properties for the semigroup $(V_{1}(t))_{t\geq 0}$ from that of $(V_{2}(t))_{t\geq 0}$ under the assumptions of Theorem \ref{theoMMG}. We refer also to \cite{scheer} for more results in this direction and to \cite[Theorem 3.1]{CL} for a much less general but simpler version of this kind of results.\end{nb}

It remains to prove that $\L_{1}$ can be split as 
$$\L_{1}=\mathcal{A}_{1}+\mathcal{B}_{1}$$
where both the operators $\mathcal{A}_{1}$ and $\mathcal{B}_{1}$ satisfy the assumptions of Theorem \ref{theoMMG}. To do so, for any $R > 0$, let 
$$\mathcal{A}_{1} f=K(\chi_{B_{R}}f), \qquad \mathcal{B}_{1}=\mathscr{L}_{1}-\mathcal{A}_{1}$$
where $B_{R}$ is the open ball in $\R^{d}$ with radius $R >0$ and center $0$ and $\chi_{B_{R}}$ denotes the indicator function of $B_{R}$. 

Our proof is very similar to the one of \cite{BCL} and is based upon the results of Appendix \ref{sec:appA}.
One has therefore, with the notations of Section \ref{sec:hard}
$$\mathcal{A}_{1}f(v)=\int_{B_{R}}k(v,w)f(w)\d w, \qquad \mathcal{B}_{1}f(v)=-\Sigma(v)f(v)+\int_{\{|w| >R\}}k(v,w)f(w)\d w$$

\begin{propo} With the above notations, for any $R > 0$, 
$$\mathcal{A}_{1}\::\:\X \to \H$$
is bounded.
\end{propo}
\begin{proof} Let $f \in \X$ be given and let $k(v,w)$ be the kernel of $K$. Recall that $K=\mathbf{K}_{\gamma}$ with kernel $k_{\gamma}(v,w).$ Using Minkowski's integral inequality (with measures $\M^{-1}(v)\d v$ and $|f(w)|\d w$) one gets easily
$$\|\mathcal{A}_{1}f\|_{\H} \leq \int_{|w|\leq R} |f(w)|\left(\IR k_{\gamma}^2(v,w)\M^{-1}(v)\d v\right)^{1/2}\d w.$$
Now, with the notations of Appendix \ref{sec:appA}, one has
$$\IR k_{\gamma}^2(v,w)\M^{-1}(v)\d v = \M^{-1}(w)\IR p_{\gamma}^2(v,w)\d v$$
and, using Lemma  \ref{lemmG}, there is some positive constant $C >0$ such that
$$\IR k_{\gamma}^2(v,w)\M^{-1}(v)\d v \leq C \M^{-1}(w)(1+|w|)^{-1} \leq C\M^{-1}(w) \qquad \forall w \in \R^3.$$
Thus,
$$\|\mathcal{A}_{1}f\|_{\H} \leq \sqrt{C}  \int_{|w|\leq R} |f(w)|\M^{-1/2}(w)\d w$$
and, since the domain of integration is bounded,  there is some positive constant $c_R >0$ such that $\|\mathcal{A}_{1}f\|_{\H} \leq c_R \|f\|_{\X}$ which proves point.\end{proof} Now, we prove that $R >0$ can be chosen in such a way that
$$\mathcal{B}_{1}f(v)=\L_{1} f(v) -\mathcal{A}_{1}f(v)=\L^+(\chi_{\{|\cdot| > R\}}f)(v)-\sigma(v)f(v)$$
satisfies the above point (ii). For any $f \in \mathcal{Y}$, set $I(f)=\IR \mathrm{sign}f(v) \mathcal{B}f(v)m^{-1}(v)\d v$.
Now, one has the following
\begin{propo} Assume that $m$ is given by \eqref{expwe} with $1-\gamma<s<1$ There exists $R > 0$ large enough so that
\begin{equation*}
    \IR \mathrm{sign}f(v) \mathcal{B}_{1}f(v)m^{-1}(v)\d v
    \leq -\beta  \|f\|_{\Y}
    \qquad \forall f \in \Y
  \end{equation*}
where $\beta=\inf_{v \in\ \R^{d}}\dfrac{\Sigma(v)}{(1+|v|)^{\gamma}}.$\end{propo}
\begin{proof} Let $f \in \Y$ be fixed.
\begin{equation*}\begin{split}
I(f)&=\IR \mathrm{sign}f(v) K(\chi_{\{|\cdot| > R\}}f)f(v)m^{-1}(v)\d v-\IR \Sigma(v)|f(v)|m^{-1}(v)\d v\\
&=\IR \mathrm{sign}f(v)m^{-1}(v)\d v \int_{\{|w|> R\}} k_{\gamma}(v,w)\d w -\IR \Sigma(v)|f(v)|m^{-1}(v)\d v  \\
&\leq \int_{\{|w|> R\}} |f(w)|H_{\gamma}(w)\d w -\sigma_0\IR (1+|v|)^{\gamma}|f(v)|m^{-1}(v)\d v
\end{split}
\end{equation*}
where we used the fact that $\Sigma(v) \geq \sigma_0 (1+|v|)^{\gamma}$ for some positive constant $\sigma_0 >0$ and set, as in Appendix \ref{sec:appA},
$$H_{\gamma}(w)=\IR k_{\gamma}(v,w)m^{-1}(v)\d v, \qquad \forall w \in \R^3.$$
Let us assume now that $m$ is an \emph{exponential weight} given by \eqref{expwe}.Then, using Proposition \ref{BCL}, there is some positive constant $C >0$ such that
$$I(f) \leq C \int_{\{|w|> R\}} |f(w)| \left(1+|w|^{\gamma-s}\right)m^{-1}(w)\d w -\sigma_0\IR (1+|v|)^{\gamma}|f(v)|m^{-1}(v)\d v.$$
In other words,
\begin{multline*} I(f)
  \leq  - \sigma_0 \int_{\{|v| \leq R\}} |f(v)| m^{-1}(v) \,\d v
  \\
 + \int_{\{|v|> R\}} |f(v)|
  \left(
    C (1 + |v|^{\gamma-s})
    -
    \sigma_0 (1+|v|)^{\gamma}
  \right)
  \, m^{-1}(v)  \,\d v.\end{multline*}
We choose now $R > 0$ such that $C(1 + |v|^{\gamma-s}) - \sigma_0
(1+|v|)^{\gamma} \leq - \sigma_0$ for all $|v| > R$ (which can be done since $s >0$), so that
\begin{equation}
  \label{eq:If-3}
  I(f)
  \leq
  - \sigma_0 \IR |f(v)| m(v)^{-1} \,\d v
  =
  - \sigma_0 \|f\|_{\mathcal{Y}},
\end{equation}
i.e. $\mathcal{B}_{1}$ satisfies \eqref{Bdiss} with $\beta=\sigma_0 $.\\

Assume now that $m$ is an algebraic weight given by \eqref{algwe}. Then, as above, using Proposition \ref{BCL}, there is some positive constant $C >0$ such that 
$$I(f) \leq C \int_{\{|w|> R\}} |f(w)| \left(1+|w|^{\gamma-2}\right)m^{-1}(w)\d w -\sigma_0\IR (1+|v|)^{\gamma}|f(v)|m^{-1}(v)\d v$$
and, repeating the above argument, one can choose $R > 0$ large enough so that $C(1 + |v|^{\gamma-s}) - \sigma_0
(1+|v|)^{\gamma} \leq - \sigma_0$ for all $|v| > R$ to see that $\mathcal{B}_{1}$ satisfies \eqref{Bdiss} with $\beta=\sigma_{0}.$\end{proof}

\end{document}